\documentclass{amsart}
\usepackage[utf8]{inputenc}
\usepackage{comment}
\usepackage{amsmath}
\usepackage{mathtools}
\usepackage{graphicx}
\usepackage{placeins}
\usepackage{subcaption}
\usepackage{amsfonts}
\usepackage{amsthm}
\usepackage{tikz-cd}
\usepackage{hyperref}
\usepackage[nameinlink]{cleveref}
\usetikzlibrary{patterns}

\usepackage{amssymb}
\usepackage{mathrsfs}	
\numberwithin{equation}{section}
\usepackage[all]{xy}
\usepackage{color}
\usepackage[left=2.5cm,right=2.5cm]{geometry}

\usepackage{bm} 

\usepackage{mathrsfs}

\usepackage{todonotes}

\DeclareMathOperator{\Gr}{Gr}

\DeclareMathOperator{\Sym}{Sym}

\newcommand{\mE}{\mathcal{E}}
\newcommand{\mF}{\mathcal{F}}
\newcommand{\mK}{\mathcal{K}}
\newcommand{\mO}{\mathcal{O}}
\newcommand{\mQ}{\mathcal{Q}}
\newcommand{\QQ}{\mathcal{Q}}

\newcommand{\mU}{\mathcal{U}}

\newcommand{\PP}{\mathbb{P}}
\newcommand{\Z}{\mathbb{Z}}
\newcommand{\N}{\mathbb{N}}

\newtheorem{lemma}{Lemma}[section]
\crefname{lemma}{Lemma}{Lemmas}
\Crefname{lemma}{Lemma}{Lemmas}

\newtheorem{proposition}[lemma]{Proposition}
\crefname{proposition}{Proposition}{Propositions}
\Crefname{proposition}{Proposition}{Propositions}

\newtheorem{thm}[lemma]{Theorem}
\crefname{thm}{Theorem}{Theorems}
\Crefname{thm}{Theorem}{Theorems}
\newtheorem{theorem}{Theorem}

\crefname{theorem}{Theorem}{Theorems}
\Crefname{theorem}{Theorem}{Theorems}

\newtheorem{coroll}[theorem]{Corollary}

\crefname{coroll}{Corollary}{Corollaries}
\Crefname{coroll}{Corollary}{Corollaries}

\newtheorem{corollary}[lemma]{Corollary}
\newtheorem*{corollary*}{Corollary}
\crefname{corollary}{Corollary}{Corollaries}
\Crefname{corollary}{Corollary}{Corollaries}

\crefname{condition}{Condition}{Conditions}
\Crefname{condition}{Condition}{Conditions}

\crefname{conclusion}{Conclusion}{Conclusions}
\Crefname{conclusion}{Conclusion}{Conclusions}

\newtheorem{definition}[lemma]{Definition}
\crefname{definition}{Definition}{Definitions}
\Crefname{definition}{Definition}{Definitions}

\newtheorem*{definition*}{Definition}

\crefname{conjecture}{Conjecture}{Conjectures}
\Crefname{conjecture}{Conjecture}{Conjectures}

\crefname{question}{Question}{Questions}
\Crefname{question}{Question}{Questions}

\crefname{claim}{Claim}{Claims}
\Crefname{claim}{Claim}{Claims}

\newtheorem{example}[lemma]{Example}
\crefname{example}{Example}{Examples}
\Crefname{example}{Example}{Examples}

\newtheorem{remark}[lemma]{Remark}
\crefname{remark}{Remark}{Remarks}
\Crefname{remark}{Remark}{Remarks}

\newtheorem*{notation*}{Notation}

\def\Hom{\operatorname{Hom}}

\def\End{\operatorname{End}}

\def\Ext{\operatorname{Ext}}
\def\rk{\operatorname{rk}}
\def\ch{\operatorname{ch}}

\DeclareMathOperator{\sHom}{\mathscr{H}\text{\kern -3pt {\calligra\large om}}\,}

\title{Weak Brill--Noether bundles on nodal Enriques surfaces}

\author{Enrico Fatighenti}
\address{\newline
Alma Mater studiorum Universit\`a di Bologna\hfill\newline
Dipartimento di Matematica\hfill\newline
Piazza di Porta San Donato 5, 40126 Bologna, Italy}
\email[E.~Fatighenti]{enrico.fatighenti@unibo.it}

\author{Filippo Papallo}
\address{\newline Università degli Studi di Genova
    \hfill\newline Dipartimento di Matematica
    \hfill\newline Via Dodecaneso 35, 16146 Genova, Italy
}
\email[F.~Papallo]{papallo@dima.unige.it}

\author{Federico Tufo}
\address{\newline
Alma Mater studiorum Universit\`a di Bologna\hfill\newline
Dipartimento di Matematica\hfill\newline
Piazza di Porta San Donato 5, 40126 Bologna, Italy}
\email[F.~Tufo]{federico.tufo96@gmail.com}

\begin{document}

\begin{abstract}
    We study the weak Brill--Noether problem for vector bundles on nodal Enriques surfaces by exploiting the classical realization of a nodal Enriques surface $X$ as a Reye congruence in the Grassmannian $\Gr(2,4)$.
    We consider restrictions of irreducible homogeneous bundles from the ambient Grassmannian and study their behavior with respect to the weak Brill--Noether property and slope stability.
   
\end{abstract}

\maketitle

\section*{Introduction}

Brill--Noether theory is a classical topic in algebraic geometry.
Historically, the problem was first developed for algebraic curves and addressed the following question: given non-negative integers $d, h$ and a smooth projective curve $C$ of genus $g$, one wants to understand whether the space of degree $d$ maps $f:C \to \PP^{h}$ is non-empty, and, in such cases, how many of these maps there exists. 
Since maps into projective spaces are determined by line bundles with many sections, the problem admits a natural cohomological reformulation: 
given a curve $C$, one asks for the existence of line bundles in $\mathrm{Pic}^{d}(C)$ 
such that $h^0(X, L) \geq h+1$; when such line bundles exist, they form a subscheme
$\mathrm{W}^h_d(C)$ inside $\mathrm{Pic}^d(C)$, classically
called a Brill--Noether locus. 
Thanks to the foundational works of Griffiths--Harris \cite{GH80}, Gieseker \cite{Gie82}, and Fulton--Lazarsfeld \cite{FL81}, among others, the theory for line bundles on curves is now well understood and
is encoded in the Brill--Noether number
    $\rho(g, r, d) := g - (r+1)(g - d + r)$:
    indeed, for a general curve $C$ of genus $g$,
    the locus $\mathrm{W}^h_d(C)$ is non-empty if
and only if $\rho \geq 0$, it has the expected dimension $\rho$ whenever it is non-empty, and it is irreducible if and only if $\rho>0$.

The problem admits two natural generalizations: one may replace the
curve $C$ by a projective variety $X$ of higher dimension, or replace
line bundles by stable vector bundles of higher rank.
However, even for smooth surfaces, the general theory is far less
developed.
A fundamental difficulty is that, unlike in the curve case, there is
no known method to control the individual cohomology groups of a
general sheaf in a moduli space purely in terms of its topological
invariants. For example, if $E$ is a vector bundle on a surface $X$, then the
Hirzebruch--Riemann--Roch theorem determines $\chi(X,E)$ from the
rank, the first Chern class, and $\mathrm{ch}_2(E)$, but it gives no
information on the individual groups $h^i(X, E)$.
Since a general semistable vector bundle
on a smooth projective curve 
has cohomology concentrated in a single degree,
it becomes meaningful to ask when such concentration has to be expected in the case of a smooth projective surface.

Fix a smooth projective surface $X$ with polarization $H$,
and invariants
\begin{equation*}
    v = \bigl(\rk(E),\, c_1(E),\, \chi(X,E)\bigr)
    \in \mathbb{Z} \oplus \mathrm{NS}(X) \oplus \mathbb{Z}\,.
\end{equation*}
The moduli space $M_{X,H}(v)$ parametrizes $S$-equivalence classes
of $H$-semistable sheaves with the given invariants $v$.
The expected cohomological behaviour of a semistable sheaf $E \in
M_{X,H}(v)$ with positive slope is to have cohomology concentrated in a single degree:
if the Euler characteristic satisfies $\chi(X,E) > 0$, one expects
$H^0(X,E) \neq 0$ and $H^1(X,E) = H^2(X,E) = 0$; on the other hand,
if $\chi(X,E)<0$, one expects $H^{1}(X,E) \ne 0$
and $H^{0}(X,E) = H^{2}(X,E) = 0$. 
The weak Brill--Noether property
makes this precise.

\begin{definition*}
Let $E$ be a coherent sheaf on a projective variety $X$.
We say that $E$ satisfies the weak Brill--Noether property
in degree $p$, denoted $\mathrm{WBN}_p$, if
\[
    H^q(X,E) = 0\,, \quad \text{for all } q \neq p.
\]
Given an ample divisor $H$ on $X$ and an irreducible component
$M \subset M_{X,H}(v)$, we say that $M$ satisfies
$\mathrm{WBN}_p$ if the general element $E \in M$ does.
\end{definition*}

Geometrically, this property gives information on which 
Brill--Noether loci define proper closed subschemes inside $M_{X,H}(v)$.
Indeed, suppose $\chi(X,E) =: \chi > 0$, so that the expected behaviour
is $\mathrm{WBN}_0$.
The locus
\begin{equation*}
    \mathrm{W}_v^{\chi} := \bigl\{\, E \in M_{X,H}(v) \;\big|\;
    h^0(X,E) \geq \chi + 1 \,\bigr\}
\end{equation*}
parametrizes sheaves whose sections are more than expected:
if $M_{X,H}(v)$ actually satisfies $\mathrm{WBN}_0$, then
for $h \ge \chi$ the locus
$\mathrm{W}_v^{h}$ is a proper subvariety of
of the moduli space and 
it becomes meaningful to study its geometrical properties,
such as its non-emptiness, codimension, irreducibility, 
and singularities: indeed,
thanks to the works
of Costa--Mir\'o-Roig \cite[Theorem~{2.3}]{CMR10} and Nugent \cite{Nug25},
there are now structural result that describe the
Brill--Noether loci $\mathrm{W}_v^{h}$ as determinantal
varieties, whose expected dimension is dictated by
the generalized $h$-th Brill--Noether number for the surface $X$:
\begin{equation*}
    \rho^{h}(v,H) := \dim M_{X,H}(v) - h(h-\chi)\,. 
\end{equation*}
Moreover, this also shows that 
the more one increases the number $h$ of independent sections,
the deeper $W_{v}^{h}$ sits in the singular strata, for
$W_{v}^{h+1} \subset \operatorname{Sing}(W_{v}^{h})$.
In the absence of $\mathrm{WBN}_0$, the locus $\mathrm{W}_v^{\chi}$
fills the entire moduli space and carries no geometric information.

The \emph{weak Brill--Noether problem} for a polarised surface
$(X,H)$, therefore, consists of classifying all Chern characters
$v$ for which some irreducible component
of $M_{X,H}(v)$ satisfies the weak Brill--Noether property.
The problem has been studied for the projective plane and for rational surfaces by G\"{o}ttsche--Hirschowitz \cite{GH98} and Coskun--Huizenga \cite{CH20}, respectively. In Kodaira dimension zero, a complete classification has been obtained for abelian surfaces and for K3 surfaces of Picard rank one by Coskun--Nuer--Yoshioka \cite{CNY25, CNY23}, via a systematic analysis of wall-crossing in the space of Bridgeland stability conditions. An extensive source for the topic is the survey \cite{CHN24}.

In this paper, we initiate the study of the weak Brill--Noether problem for nodal Enriques surfaces\footnote{Despite the name, these Enriques surfaces are smooth.}. 
While Bridgeland stability techniques are available in this setting, we adopt a different approach by studying vector bundles representing
points of $M_{X,H}(v)$, for certain choices of $v$.
To be more precise, we exploit the classical realization of a nodal Enriques surface $X$ as a Reye congruence, namely as a subvariety of the Grassmannian $\Gr(2,4)$, and construct families of semistable vector bundles on $X$ by restricting homogeneous bundles from $\Gr(2,4)$. Every irreducible homogeneous vector bundle on $\Gr(2,4)$ can be written in the form $\Sym^n\mathcal{Q} \otimes \Sym^m\mathcal{U}(s)$, for $\mQ$ and $\mU$, the quotient and the tautological bundle of $\Gr(2,4)$. We denote by $E_{n,m,s}$ its restriction to $X$. This construction provides an explicit relation between the parameters $(n,m,s) \in \mathbb{N} \times \mathbb{N} \times \mathbb{Z}$ and the topological invariants $(\rk(E_{n,m,s}), c_1(E_{n,m,s}), \chi(E_{n,m,s}))$, and leads to concrete numerical regions where the weak Brill--Noether property holds.

Our main result identifies such triples $(n,m,s) \in \N \times \N \times \Z$ for which $E_{n,m,s}$ satisfies the weak Brill--Noether property.
\begin{theorem}[\cref{thm:wbn}]\label{thm:A}
Let $X \subset \Gr(2,4)$ be a nodal Enriques surface, and set $E_{n,m,s} := \Sym^n\mathcal{Q} \otimes \Sym^m\mathcal{U}(s)\vert_X$. Then:
\begin{itemize}
    \item If 
    \[
    (n,m,s) \in 
    \{s \geq \max(1,\, m-1,\, m+1-n)\}
    \setminus \{(1,1,1), (2,2,1)\},
    \]
    then $E_{n,m,s}$ satisfies $WBN_0$.
    
    \item If 
    \[
    (n,m,s) \in 
    \begin{aligned}[t]
        &\{ \max(2,\, m-1) \leq n \leq m+3,\ s = m-n+1\}\cup \\
        &\cup\ \{m \geq 2,\ \max(1,\, m-2) \leq n \leq m-1,\ s = 1\},
    \end{aligned}
    \]
    then $E_{n,m,s}$ satisfies $WBN_1$.
    
    \item If 
    \[
    (n,m,s) \in 
    \begin{aligned}[t]
        &\{ s \leq -n\} \cup\ \{n = 0,\ m \geq 2,\ s = 1\}
        \setminus \{(0,0,0), (1,1,-1)\},
    \end{aligned}
    \]
    then $E_{n,m,s}$ satisfies $WBN_2$.
\end{itemize}
\end{theorem}

The numerical condition of \Cref{thm:A} can be visualized
in Figures~\ref{figure:m-small}-\ref{figure:m-big}, right after the Bibliography. By upper semicontinuity of cohomology in flat families, the weak Brill--Noether property extends to entire irreducible components of the corresponding moduli spaces. Therefore, we obtain the following description of moduli spaces as an immediate consequence. In the following, $H$ denotes the restriction of Pl\"ucker polarization (see \cref{rmkPlR}).
\begin{coroll}[\Cref{mukVecmns}, \Cref{cor:wbn-moduli}]\label[coroll]{mainCor}
    For $(n,m,s) \in \N \times \N \times \Z$, 
    set $v=(r,c,\chi)$ to be the vector with entries
    \begin{align*}
        r &= (n+1)(m+1)\,;\ \ c = r \, \frac{n-m+2s}{2} H\,; \\
        \chi &= r \left( 1+5s(n-m + s) + \frac{n^2-3n-5nm}{2}+ \frac{7m^{2}-m}{6} \right)\,.
    \end{align*}
    If $(n,m,s)$ is as in \Cref{thm:A}, then $M_{X,H}(v)$ has
    a whole irreducible component whose general element
    has at most one non-vanishing cohomology group.
\end{coroll}
 
Although our classification is partial, it provides explicit families illustrating the asymptotic behaviour of cohomology for large numerical invariants.
Thanks to the cohomology computations carried out in the first part, we can explore the problem of the stability of the bundles restricted to the surface $X$.
\begin{theorem}[\cref{thm:simple}]\label{thm:B}
    Let $X \subset \Gr(2,4)$ be a nodal Enriques surface. 
    If
\[
0 \le n \le 3 \ \text{and} \ m = 0\,, \quad \text{or} \quad 0 \le n \le 2 \ \text{and} \ m = 1\,,
\]
    then the bundle $E_{n,m,s} := \Sym^n\mathcal{Q} \otimes \Sym^m\mathcal{U}(s)\vert_X$ is simple for any $s$. 
    The invariants and cohomological properties of these bundles are listed in
\Cref{tabEnd}.
\end{theorem}

\subsection*{Plan of the paper}
The paper is organized as follows. In \Cref{sec:reye+EN}, we review the geometry of nodal Enriques surfaces, their construction as degeneracy loci of maps of homogeneous vector bundles on $\Gr(2,4)$, and the necessary results from representation theory. In \Cref{sec:stability}, we study stability of the vector bundles $E_{n,m,s}$ with respect to a natural polarization $H$. \Cref{sec:wbn} contains the cohomological computations and a partial classification of the Chern characters satisfying the weak Brill--Noether property. Finally, in \Cref{sec:simple}, we study the slope stability of the constructed bundles, and we exhibit several explicit stable examples.

\subsection*{Notation} We work over $k=\mathbb{C}$, the field of complex numbers. $V_n$ will denote a $n$-dimensional vector space, 
and $\Gr(k,n)$ the Grassmannian parametrizing its $k$-dimensional subspaces,
with tautological bundle $\mU$, of rank $k$, and quotient bundle $\mQ$, of rank $n-k$ and ample determinant. If $\mE$ and $\mF$ are vector bundles on an algebraic variety $X$, and $s:\mE\to\mF$ a map of vector bundles, then we define $\mathrm{D}_k(s):=\{x\in X|\rk{s_x}\leq k\}$ the locus where $s$ has rank at most $k$. If $H$ is a fixed polarization on a surface $X$ and $\mE$ is a vector bundle on $X$, we will denote by $\mu_{H}(\mE) := c_{1}(\mE) \cdot H/\rk(\mE)$
the slope of $\mE$ with respect to $H$.
By totally acyclic sheaf $\mF$ on $\Gr(2,4)$
we mean $\mF$ such that $H^{*}(\mF,\Gr(2,4))=0$.

\subsubsection*{Acknowledgment}
 We want to thank H. Nuer for introducing us to the problem 
 and sharing helpful comments on the preliminary draft of this paper. 
This research has been partially funded by the European Union - NextGenerationEU under the National Recovery and Resilience Plan (PNRR) - Mission 4 Education and research
- Component 2 From research to business - Investment 1.1 Notice Prin 2022 - DD N. 104 del 2/2/2022,
from title “Symplectic varieties: their interplay with Fano manifolds and derived categories”, proposal
code 2022PEKYBJ – CUP J53D23003840006.
 The authors are members of INDAM-GNSAGA.

\section{On nodal Enriques surfaces}\label{sec:reye+EN}

In this section we introduce nodal Enriques surfaces and discuss the computations that can be carried out on them by means of representation theory. We begin with the following definition.

\begin{definition}
    An Enriques surface is a smooth projective surface with $H^{1}(X,\mO_{X})=0$ and a $2$-torsion non-trivial canonical divisor $K_{X}$, i.e. $2K_{X} \sim 0$. An Enriques surface is called nodal if it contains a rational curve of self-intersection $-2$.
\end{definition}

Nodal Enriques surfaces form a divisor in the $10$-dimensional moduli space of polarized Enriques surfaces. By \cite{MMV24}, it is known that any such $X$ can be embedded in the Grassmannian $\Gr(2,4)$ as a Reye congruence of lines $R(W)$, that is an irreducible 
subvariety of $\Gr(2,4)$ parametrizing those
lines in $\PP^{3}$ contained in (at least) 
two distinct quadrics of a three dimensional linear system $W \subset |\mO_{\PP^{3}}(2)|$.

Any Reye congruence of lines, by \cite{KI15}, can be described as $\mathrm{D}_{2}(\varphi)$, with
\begin{equation*}
      \varphi: \mO^{\oplus 4} \longrightarrow \Sym^{2}\mU^{\vee}\,.
\end{equation*} 

The structure sheaf $\mO_X$ of a nodal Enriques surface is resolved by the Eagon--Northcott complex of vector bundles on $\Gr(2,4)$, which in this case is:
\begin{equation} \label{eqn:Eagon--Northcott}
  0 \to \Sym^2 \mU(-3) \to \mO(-3)^{\oplus 4} \to \mO \to \mO_X \to 0
\end{equation}
 
\begin{remark}\label[remark]{rmkPlR}
The restriction of the Pl\"ucker line bundle  $\mO_{Gr(2,4)}(1)|_X$, whose divisor class is denoted by $H$ in the rest of the paper, is often called Fano--Reye polarization.
\end{remark}

In this work, we study the restriction to $X$ of irreducible homogeneous vector bundles. We fix the notation
\[
\mE_{\alpha|\beta}:=\Sigma_\alpha \mQ \otimes \Sigma_\beta \mU,
\]
where $\alpha=(\alpha_1, \alpha_2)$ and $\beta= (\beta_1, \beta_2)$ are partitions of length 2 (in particular $\alpha_1 \geq \alpha_2 \geq 0$, and similarly for $\beta$), and $\Sigma_\alpha \mQ$ denotes the corresponding Schur functor applied to $\mQ$ (and similarly to $\mU$). The cohomological behaviour of these bundles is completely controlled by the Borel--Weil--Bott theorem. Here we state the result for $\Gr(2,4)$.

\begin{thm}\cite[Corollary 4.1.9]{Wey03}\label[thm]{thm:BWB}
    Let $\rho = (3,2,1,0)$ be the Weyl vector of $\mathfrak{gl}_4$. 
    For a weight $\lambda = (\alpha_1, \alpha_2, \beta_1, \beta_2) \in \N^4$, consider the vector $\lambda + \rho$.
    \begin{enumerate}
        \item If $\lambda + \rho$ has repeated entries, then $H^i(\Gr(2,4), \mE_{\lambda}) = 0$ for all $i \ge 0$.
        \item Otherwise, there exists a unique element $w$ in the Weyl group $\mathfrak{S}_4$ such that the weight $\mu = w(\lambda + \rho) - \rho$ is dominant (i.e. strictly decreasing). If $l(w)$ denotes
        the length of $w$, then 
        $$H^{l(w)}(\Gr(2,4), \mE_{\lambda}) \cong \Sigma_\mu \mathbb{C}^4\,,$$ and all other cohomology groups vanish.
    \end{enumerate}
\end{thm}

Notice that any homogeneous irreducible vector bundle $\mE_{\alpha \mid \beta}$ on $\Gr(2,4)$ can be expressed in a simpler form. The same applies to its restriction to $X$:

\begin{lemma}\label[lemma]{Enms}
Let $E$ be the restriction to $X$ of a homogeneous, irreducible vector bundle from $\Gr(2,4)$. Then $E$ can be rewritten as
\[
E_{n,m,s}:= \Sym^n \mQ|_X \otimes \Sym^m \mU|_X \otimes \mO_X(s),
\]
with $(n,m,s) \in (\N \times \N \times \Z)$.

\end{lemma}
\begin{proof}
As we discussed, each homogeneous and irreducible vector bundle on $\Gr(2,4)$ can be given as $\mE_{\alpha|, \beta}:=\Sigma_\alpha \mQ \otimes \Sigma_\beta \mU$, hence determined by a quadruple of integers $\gamma:=(\alpha_1, \alpha_2, \beta_1, \beta_2)$ with $\alpha_1 \geq \alpha_2 \geq 0, \beta_1 \geq \beta_2 \geq 0$. On the other hand, if we set $l:=\min(\alpha_2, \beta_2)$, and denote with $\underline{l}:=(l,l,l,l)$ we have that 
\[
\mE_\gamma \cong \mE_{\gamma-\underline l};
\]
since shifting the whole sequence by the same integer is equivalent to tensor the bundle with $\det(\mQ)^{\otimes l } \otimes \det(\mU)^{\otimes l} \cong \mO$. After the shift, each entry in $\gamma-\underline l$ is non-negative. In other words, $\gamma -\underline l $ can be rewritten either as $(n+s, s, m,0)$ or $(n,0,m-s, -s)$. The result follows.
\end{proof}

By the Eagon--Northcott complex~\eqref{eqn:Eagon--Northcott},
we see each $E_{n,m,s}$ can be resolved by three
homogeneous vector bundles on $\Gr(2,4)$,
not all of which will be irreducible. Decomposition in irreducible factors of tensor products of such bundles can be handled with the well-known Littlewood--Richardson rule, see for example \cite[(A.8)]{FH91}.
By combining this, \cref{thm:BWB} and \eqref{eqn:Eagon--Northcott}, one can systematically compute the cohomology of bundles on $X$ that are restricted from an irreducible homogeneous bundle on $\Gr(2,4)$.

We end this section with a couple of definitions that will be used in the paper.
Any Enriques surface $X$ (not necessarily nodal) has the structure of an elliptic fibration.
\begin{definition}
    A half-fibre is a connected, isotropic nef divisor $F$ with $h^{0}(F)=1$.
\end{definition}
The name comes from the fact that $\lvert 2F \rvert$ is an elliptic pencil \cite[Proposition~{2.2.8}]{CDL25}. 
By \cite[Proposition~{2.3.4}]{CDL25}, 
any Enriques surface has at least one half-fibre
and in fact, when $X$ is a Reye congruence, there
exists a sequence of $10$ half fibres $F_{1}, \dots, F_{10}$ 
such that their sum is linearly equivalent to $3H$ and $H \cdot F_{i} = 3$, for each $i$ (see \cite[Theorem~{7.7.2}]{DK25}).
The dual of the restricted tautological bundle $\mU^{\vee}\vert_{X}$, often called Reye bundle in the literature, is related to these divisors via an extension:
\begin{thm}\cite[Theorem~{7.8.3}]{DK25}\label[thm]{thm:U-ext}
    Let $H$ be the Fano--Reye polarization on $X$ and $F$ a half-fibre in the $10$-sequence above.
    Then there is a non-split short exact sequence
    \begin{equation*}
        0 \longrightarrow \mO_{X}(F) \longrightarrow \mU\vert_{X}^{\vee}
        \longrightarrow \mO_{X}(H-F) \longrightarrow 0\,.
    \end{equation*}
\end{thm}

For $D$ any ample divisor on $X$, one defines
    $\phi(D) := \min \{ D \cdot F \, | \, F \text{ half-fibre }\}$.
The Fano--Reye polarization satisfies
\begin{equation*}
    H^{2} = 10\,, \quad \phi(H) = 3\,.
\end{equation*}

\section{Stability of tautological bundle(s)}\label{sec:stability}

In this section, we want to study the slope stability of $\mE_{\alpha|\beta}\vert_{X}$
with respect to the Fano--Reye polarization $H$.
It is a well-known fact that irreducible homogeneous vector bundles are stable on Grassmannians; see, for instance, \cite{Ram66}. This, of course, says nothing a priori about the stability properties on the restriction of such bundles. 
However, we prove here that this is the case for the restriction of the tautological and quotient bundle on $X$.

				\begin{lemma} \label[lemma]{lem:stability} 
					The restrictions $\mQ\vert_{X}$ and $\mU\vert_{X}$ 
					are slope stable bundles with respect to $H$ on $X$.
					\begin{proof}
                    Stability for $\mQ|_X$ is a classical fact, and the proof can be found in e.g. \cite[Corollary~{4}]{DR91}.
                    
						The case $\mU\vert_{X}$ requires some more work. 
						Since dualizing a vector bundle preserves stability, 
						we study $\mU\vert_{X}^{\vee}$, which has $c_{1}(\mU{\vert}_{X}^{\vee}) = H$
						and is globally generated.
						By using the resolution \eqref{eqn:Eagon--Northcott}, one can compute
						\begin{align*}
							h^{0}(\mU\vert_{X}^{\vee}) &= \chi(\mU\vert_{X}^{\vee}) \\
							&= \chi(\mU^{\vee}) 
							- 4\chi(\mU^{\vee}(-3)) 
							+ \chi(\mU^{\vee} \otimes \Sym^{2}\mU(-3)) \\
							&= h^{0}(\mU^{\vee}) = 4\,,
						\end{align*}
						so that $c_{2}(\mU{\vert}_{X}^{\vee}) 
						= 7 - \chi(\mU{\vert}_{X}^{\vee}) = 3$ by Hirzebruch--Riemann--Roch.
						
						We argue by contradiction, assuming there exists
						a destabilizing sequence
						 $0 \to A \to \mU\vert_{X}^{\vee} \to B \to 0$,
						with $A$ and $B$ line bundles
						such that
						\begin{equation*}
							\mu_{H}(A) = c_{1}(A) \cdot H \ge \mu_{H}(\mU\vert_{X}^{\vee}) 
							= \frac{H^{2}}{2} = 5\,.
						\end{equation*}
						To ease notation, set $a := c_{1}(A)$ and $b := c_{1}(B) = H - a$.
						From $c_{2}(\mU{\vert}_{X}^{\vee}) = a \cdot b$ one deduces
						\begin{equation*}
							3 = a \cdot (H - a) = \mu_{H}(A)-a^{2} 
							\quad \implies \quad \mu_{H}(A) = a^{2} + 3\,,
						\end{equation*}
						so in particular $a^{2} \ge 2$.
						On the other hand $H = a + b$, so
						\begin{equation*}
							10 = H^{2} = a^{2} + b^{2} + 2 a \cdot b \ge a^{2} + 6\,,
						\end{equation*}
						where we use $b^{2} \ge 0$, because $B$ is globally generated.
						Thus, we have two cases:
						\begin{enumerate}
							\item $a^{2} = 4$ and $\mu_{H}(A) = 7$. Since $\phi(H) = 3$,
                            there exists $F$ a half-fibre for which $\mu_{H}(\mO_{X}(F))=3$,
							hence $\Hom(A,\mO_{X}(F)) = 0$ by stability; since both $A$ and $\mO(H-F)$ are stable bundles of the same slopes, then Schur's Lemma~\cite[Proposition~{1.2.7}]{HL10} implies
                            the composition
							$A \to \mU{\vert}_{X}^{\vee} \to \mO_{X}(H-F)$
							must be an isomorphism, but this
							contradicts the fact that the sequence 
							\Cref{thm:U-ext} is non-split.
							
							\item $a^{2}=2$ and $\mu_{H}(A) = 5$. 
							Then also $b^{2}=2$ and $\mu_{H}(B) = 5$; in particular,
							$B$ is nef and big, so from 
                            \cite[Theorem~{2.1.16}]{CDL25} it follows
							\begin{equation*}
								2 = \frac{b^{2}}{2} + 1 = \chi(B) = h^{0}(B)\,,
							\end{equation*}
							But then $\mu_{H}(B) < 2\phi(H)$ contradicts \cite[Proposition~{2.3}]{MMV24}.  \qedhere
						\end{enumerate}
					\end{proof}
				\end{lemma}

				 As an easy corollary, we obtain the following:
				
\begin{corollary}\label[corollary]{cor:polystability}					
Let $H$ be the Fano--Reye polarization on $X$. 
					For all $\alpha,\beta $, the bundle 
					\begin{equation*}
						\mE_{\alpha|\beta}|_X:= \Sigma_\alpha \mQ|_X \otimes \Sigma_\beta \mU|_X
					\end{equation*}
					is $\mu_{H}$-polystable;
                    in particular, $\mE_{\alpha|\beta}|_X$ is Gieseker $H$-semistable.
\end{corollary}	
\begin{proof}
 We have proven in \Cref{lem:stability} that both $\mU|_X$ and $\mQ|_X$ are stable. Hence, by \cite[Corollary~{3.2.11}]{HL10}, all their tensor products are polystable. Since every $\mE_{\alpha|\beta}$ appears in a tensor product decomposition by the Littlewood--Richardson rule, it follows that  $\mE_{\alpha|\beta}$ is $\mu_{H}$-polystable, which means
 direct sum of $\mu_{H}$-stable sheaves; as $\mu_{H}$-stability
 implies Gieseker $H$-stability of each factor, then $\mE_{\alpha|\beta}$
 is in particular Gieseker $H$-semistable.	
\end{proof}

We now write down explicitly the Chern character of $E_{n,m,s}$ in terms of the Chern character of $\mQ\vert_{X}$ and $\mU\vert_{X}$.
\begin{proposition}\label[proposition]{mukVecmns}
    The vector bundle $E_{n,m,s}$ on $X$ has Chern character
    \[
    \ch(E_{n,m,s})= (n+1)(m+1) \, \left(1, \frac{n-m+2s}{2} H, 5s(n-m + s) +\frac{n^2-3n-5nm}{2}+ \frac{7m^{2}-m}{6}\right)\,.
    \]
\end{proposition}
\begin{proof}
    Since the Chern character is multiplicative, i.e. $\ch(A \otimes B) = \ch(A)\ch(B)$,
    it is enough to know $\ch(\Sym^{n}\mQ\vert_{X}), \ch(\Sym^{n}\mU\vert_{X})$
    and $\ch(\mO_{X}(s))$.
    From the computations in the proof of \Cref{lem:stability}, we know
    \begin{equation*}
        \ch(\mQ\vert_{X}) = (2, H, -2)\,, \quad \ch(\mU\vert_{X}) = (2, -H, 2)\,;
    \end{equation*}
    by the formulae for symmetric powers (see for instance \cite[Appendix~{B.1}]{DFO25}),
    for a rank $2$ bundle $E$ if holds
    \begin{align*}
        \rk(\Sym^{n} E) &= n+1\,, \\
        c_{1}(\Sym^{n} E) &= \frac{n(n+1)}{2} c_{1}(E)\,, \\
        \ch_{2}(\Sym^{n} E) &= \frac{n(n+1)(n-1)}{12} c_{1}^{2}(E) + \frac{n(n+1)(n+2)}{6}\ch_{2}(E)\,,
    \end{align*}
    one can then compute
    \begin{equation*}
        \ch(\Sym^{n} \mQ\vert_{X}) = (n+1) \left(1, \frac{n}{2} H, \frac{n^2-3n}{2} \right)\,,
        \quad \ch(\Sym^{m} \mU\vert_{X}) = (m+1) \left(1, -\frac{m}{2} H, \frac{7m^{2}-m}{6} \right)\,,
    \end{equation*}
    hence, putting everything together, one obtains
    \begin{align*}
        \ch(E_{n,m,s}) &= \ch(\Sym^{n} \mQ\vert_{X})\ch(\Sym^{m} \mU\vert_{X})\ch(\mO_{X}(s)) \\
        &= (n+1)(m+1) \, \left(1, \frac{n-m}{2} H, \frac{n^2-3n}{2}+ \frac{7m^{2}-m}{6} - \frac{5nm}{2}\right)(1,sH, 5s^2) \\
        &= (n+1)(m+1) \, \left(1, \frac{n-m+2s}{2} H, 5s(n-m + s) +\frac{n^2-3n-5nm}{2}+ \frac{7m^{2}-m}{6}\right)\,.
    \end{align*}
    Let us remark that, by applying Hirzebruch--Riemann--Roch
    for $X$, one  may compute the Euler characteristic of $E$ 
    by $\chi(X,E) = \ch_{2}(E) + \rk(E)$, which yields the statement of \cref{mainCor}.
\end{proof}

\section{Weak Brill--Noether property}\label{sec:wbn}

Given $n,m \in \N$ and $s \in \Z$, one can easily determine for which $p$ the homogeneous bundle
\begin{equation*}
   \mE_{n,m,s} := \Sym^{n}\mQ \otimes \Sym^{m}\mU(s)
\end{equation*}
satisfies $WBN_{p}$.
Indeed, thanks to \cref{thm:BWB}, this amounts to studying a simple inequality. Recall first that, since $\mE_{n,m,s}$ is homogeneous, then $$H^l(\Gr(2,4), \mE_{n,m,s}) \neq 0\,,$$ for at most one $l$.
\begin{lemma}\label[lemma]{lemma:coomologico}
    The cohomology of the bundle $H^l(\Gr(2,4),\Sym^{n}\mQ \otimes \Sym^{m}\mU(s))$ is non-zero if and only if
    \begin{equation*}
       l =
        \begin{cases}
            0 \,, \quad &\text{if } s > m-1 \,;\\
           1 \,, \quad &\text{if } \max\{m-n-2, -2\} < s < m-1 \,;\\
          2 \,, \quad &\text{if } \min\{m-n-2,-2\} < s < \max\{m-n-2,-2\} \,;\\
           3 \,, \quad &\text{if } -n-3 < s < \min\{m-n-2, -2\}\,;\\
           4 \,, \quad &\text{if } s<-n-3 \,;\\
        \end{cases}
    \end{equation*}
    In particular, the above bundle is totally acyclic if and only if an equality occurs in the above set of inequalities.  The inequalities are graphically represented in Figure \ref{figureChambers}.
\end{lemma}
\begin{proof}
    As shown in the proof of \Cref{Enms}, the bundle $\mE_{n,m,s}$ 
    is determined by a quadruple $\gamma=(n+s,s,m,0)$
    or $\gamma=(n,0,m-s,-s)$, depending on the sign of $s \in \Z$;
    for the computations that will follow, 
    we may allow $\gamma$ to have also negative entries,
    so that we may assume $\gamma=(n+s,s,m,0)$.
    After adding the weight vector $\delta = (3,2,1,0)$ to $\gamma$, one necessarily has $n+s+3>s+3$, 
    and \cref{thm:BWB} guarantees that $F$ is totally acyclic if and only if there exist two identical entries in $\gamma+\delta$, i.e.\ , at least one of the following equalities holds:
\[
    n+s+3=m+1\,, \quad \text{or} \quad s+2 = m+1\,, 
    \quad \text{or} \quad n+s+3 = 0\,, \quad \text{or} \quad s+2=0\,.
\]
    These define the dashed lines in \Cref{figureChambers}.
    Otherwise, the entries of $\gamma + \delta$ can be reordered by a minimum of $p$
    adjacent transpositions; that is exactly the $p$ for which $F$ has $WBN_{p}$. By fixing $0 \le p \le 4$,
    one obtains the conditions in the statement.
\end{proof}

			\begin{figure}
			\centering
				\begin{tikzpicture}
					\fill[black!10] (0,5) -- (0,3) -- (7,3) -- (7,5) -- (0,5);
					\fill[pattern=crosshatch dots] (0,-3) -- (3,-6) -- (0,-6) -- (0,-3);
					\fill[pattern=north east lines] (0,2) -- (0,-2) -- (4,-2) -- (0,2);
					\fill[pattern=north east lines] (7,-5) -- (4,-2) -- (7,-2) -- (7,-5);
					\draw[ultra thick, ->] (0,-6) -- (0,5) node[left]{$s$};
					\draw[ultra thick, ->] (0,0) -- (7,0) node[below]{$n$};
					\draw[thick, dashed] (0,-2) node[left]{$-2$} -- (7,-2);
					\draw[thick, dashed] (0,3) node[left]{$m-1$} -- (7,3);
					\draw[thick, dashed] (0,2) -- (7,-5) node[sloped, very near start, above]{\tiny{$n+s=m-2$}};
					\draw[thick, dashed] (0,-3) -- (3,-6) node[sloped, above, near start]{\tiny{$n+s=-3$}};
					\fill[black!50] \foreach \x in {0,1,2,3,4,5,6}{
						\foreach \y in {-5,-4,-3,-2,-1,0,1,2,3,4}{
					(\x,\y) circle(3pt)
					}
					};
					\draw (3.5,4.5) node {\small{$WBN_{0}$}};
					\draw (4.5,0.5) node {\small{$WBN_{1}$}};
					\draw (1.5,-1.5) node [fill=white]{\small{$WBN_{2}$}};
					\draw (6.5,-2.5) node [fill=white]{\small{$WBN_{2}$}};
					\draw (3.5,-4.5) node [fill=white]{\small{$WBN_{3}$}};
					\draw (1.35,-5.5) node [fill=white]
					{\small{$WBN_{4}$}};
				\end{tikzpicture}
				\caption{ 
				The weak Brill--Noether regions for the bundle 
                $\Sym^{n}\mQ \otimes \Sym^{m}\mU \otimes \mO_{G}(s)$,
                described in \Cref{lemma:coomologico}.}
                \label{figureChambers}
				
			\end{figure}
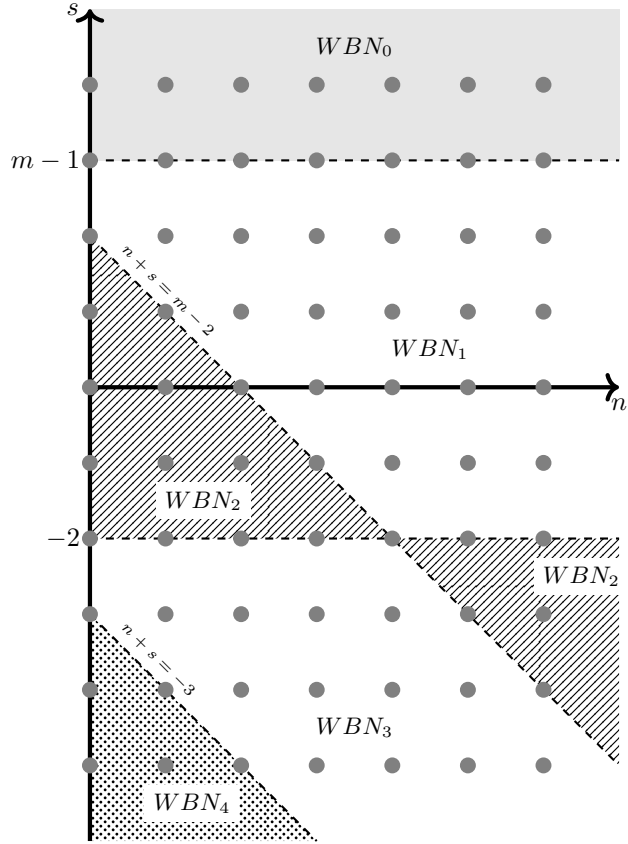

\subsection{WBN property and associated chambers decomposition}\label{subsection:wbn}
In this subsection, we find a set of sufficient conditions on $(n,m,s)$ to guarantee $E_{n,m,s}$ has the WBN property.
Before giving a list, we split the result in steps.
    
Given $E_{n,m,s}$, the Eagon--Northcott complex \eqref{eqn:Eagon--Northcott} yields the resolution
    \begin{equation}\label{eqn:res}
        0 \to \mE^{-2} \to (\mE^{-1})^{\oplus 4} \to \mE^{0} \to E_{n,m,s} \to 0\,,
    \end{equation}
    where we set $\mE^{0} := \mE_{n,m,s} =\Sym^n \mQ\otimes \Sym^m \mU(s)$,
    $\mE^{-1} = \mE_{n,m,s-3}$ and, 
    by Pieri's Rule\cite[Appendix~A.1]{FH91}, the last term
    of the resolution decomposes as
    \begin{equation*}
        \mE^{-2} = 
        \begin{cases}
            \mE_{n,2,s-3} \,, \quad &\text{if } m=0\,; \\
            \mE_{n,3,s-3} \oplus \mE_{n,1,s-4} \,, \quad &\text{if } m=1\,; \\
            \mE_{n,m+2,s-3} \oplus \mE_{n,m,s-4} \oplus \mE_{n,m-2,s-5}\,, \quad &\text{if } m \ge 2\,. 
        \end{cases}
    \end{equation*}
    Hence, \Cref{lemma:coomologico} provides a set of inequalities
    which allows to deduce for which $p$ the term $\mE^{i}$
    has $WBN_{p}$. Here we state some useful results, to help distinguish different behaviours.

\begin{lemma}\label[lemma]{lemma:coomologico.0}
    If $(n,m,s) \ne (1,1,1),(2,2,1)$ is such that 
    $s \ge \max\{1, m-1,m+1-n\}$,
    then:
    \begin{enumerate}
        \item $\mE^{0}$ 
        satisfies $WBN_{0}$;
        
        \item $\mE^{-1}$ 
        satisfies either $WBN_{0}$ or $WBN_{1}$;

        \item $\mE^{-2}$ can have at most two
        non-zero cohomology groups $H^{l}(\mE^{-2})$
        and $H^{l+1}(\mE^{-2})$, with $0 \le l \le 1$.
    \end{enumerate}
    \begin{proof}
        By assumption, $s \ge m-1$, hence $\mE^{0}$ has
        at most degree zero cohomology by \Cref{lemma:coomologico}.
        As $s-3 \ge \max\{-2, m-4,m-n-2\} \ge \max\{-2,m-n-2\}$, then the statement for $\mE^{-1}$ is clear.

        We divide the study for $\mE^{-2}$ in cases, according
        to the number of its irreducible components:
        \begin{itemize}
            \item if $m=0$, then $\mE^{-2}=\mE_{n,2,s-3}$,
            hence $s-3 \ge \max\{-2,-n-2\} = -2 \ge \min\{-n,-2\}$,
            thus $\mE^{-2}$ is either acyclic
            or has $WBN_{p}$, for $0 \le p \le 2$;

            \item if $m=1$, then 
            $\mE^{-2}=\mE_{n,3,s-3} \oplus \mE_{n,1,s-4}$.
            By the inequality
            \begin{align*}
                s - 3 &\ge \max\{-2,-n-1\} \ge -2 \ge \min\{-2,-n+1\}\,,
            \end{align*}
            one deduces that the term $\mE_{n,3,s-3}$ can have, 
            at most, cohomology up to degree $2$. 
            By \Cref{lemma:coomologico}, the same holds true for 
            $\mE_{n,1,s-4}$ if and only if $s-4 \ge \min\{-2,-n-1\}$,
            so we check in which cases it holds:
            if $n=0$, then $s-4 \ge \max\{-3, -n-2\}=-2=\min\{-2,-n-1\}$,
            so the claim is true; in case
            $n \ge 2$, then $s-4 \ge \max\{-3, -n-2\}=-3 \ge -n-1$,
            so the claim holds as well; it remains the case $n=1$, 
            in which case the claim holds if and only if $s \ge 2$.
            Hence, the only case the claim fails is $(n,m,s) = (1,1,1)$, 
            for which $\mE^{-2}$
            satisfies $WBN_{3}$: indeed, for this triple
            \begin{equation*}
                s-3=-2 = \min\{3-1-2, -2\}\,, \quad 
                -n-3= -4 < s-4 = -3 < \min\{1-1-2, -2\}\,,
            \end{equation*}
            hence \Cref{lemma:coomologico} implies
            $H^{*}(\mE^{-2}) \simeq H^{3}(\mE_{1,1,-3}) \ne 0$.
            This shows $\mE^{-2}$ is either acyclic
            or has $WBN_{p}$, for $0 \le p \le 2$, 
            for all $(n,m,s) \ne (1,1,1)$, as in the statement;

            \item if $m \ge 2$, then 
            $\mE^{-2}=\mE_{n,m+2,s-3} \oplus \mE_{n,m,s-4}
            \oplus \mE_{n,m-2,s-5}$. 
            Assume $(n,m,s) = (2,2,1)$, so that
            $\mE^{-2}=\mE_{2,4,-2} \oplus \mE_{2,2,-3}
            \oplus \mE_{2,0,-4}$.
            From \Cref{lemma:coomologico} one deduces
            \begin{align*}
                H^{*}(\mE^{-2}) \simeq H^{3}(\mE_{2,2,-3}) \ne 0\,,
            \end{align*}
            so $\mE^{-2}$ has $WBN_{3}$. 
            In all the other cases, one can check by
            an easy, yet cumbersome, case-by-case analysis 
            that the following inequalities hold
            \begin{align*}
                s-3 &\ge \min\{m-n,-2\}\,, \\
                s-4 &\ge \min\{m-n-2,-2\}\,, \\
                s-5 &\ge \min\{m-n-4,-2\}\,,
            \end{align*}
            from which the statement follows.
        \end{itemize}
    \end{proof}
\end{lemma}

\begin{lemma}\label[lemma]{lemma:coomologico.1}
    If $(m,n,s)$ satisfies either one of the following two conditions:
    \begin{itemize}
        \item[a)] $\max\{2,m-1\} \le n \le m+3$ and $s = m-n+1$, or
        \item[b)] $(n,m,s)=(n,m,1)$, with 
        $m \ge 2$ and $\max\{1,m-2\}\le n \le m-1$,
    \end{itemize}
    then:
    \begin{enumerate}
        \item $\mE^{0}$ 
        satisfies $WBN_{1}$;

        \item $\mE^{-1}$ is totally acyclic;

        \item $\mE^{-2}$ satisfies $WBN_{3}$.
    \end{enumerate}
    \begin{proof}
        The hypothesis on $s$ guarantee that either 
        $s-3=m-n-2$ or $s-3=-2$,
        in which cases \Cref{lemma:coomologico} 
        implies $\mE^{-1}$ is totally acyclic.
        Restrictions on $n$ ensure $\max\{ m-n-2, -2\} \le s \le m-1$,
        so by \Cref{lemma:coomologico} it follows
        $\mE^{0}$ is either totally acyclic or has $WBN_{1}$.
        We now study the components of the $\mE^{-2}$ term: notice that conditions in a) yield
        $m-n-2 > -2$ if and only
        if $m \ge 3$ and $n=m-1$, in which case $s-3=m-n-2=-1=m-n$
        implies that the factor $\mE_{n,m+2,s-3}$ in $\mE^{-2}$ is totally acyclic;
        in all the other cases, it holds
        \begin{equation*}
            -n-3 < s-3=m-n-2 \le \min\{m-n,-2\}\,,
        \end{equation*}
        hence $\mE_{n,m+2,s-3}$ has $WBN_{3}$;
        the restriction on $n$ guarantees 
        \begin{equation*}
            -n-3 \le s-4 = m-n-3 \le \min\{m-n-2,-2\}\,,
        \end{equation*}
        hence $\mE_{n,m,s-4}$ is either totally acyclic or has $WBN_{3}$,
        while the component $\mE_{n,m-2,s-5}$ (if it exists) is totally acyclic 
        because $s-5=(m-2)-n-2$.
        In case b), the bundle $\mE^{-2}$ splits into three components;
        notice that $s-3=-2$, hence the factor $\mE_{n,m+2,-2}$ is totally acyclic, and the other two factors of $\mE^{-2}$
        may have at most $WBN_{3}$, for we have inequalities
        \begin{align*}
            -n-3 \le -4 < -3 = s-4 \le \min\{m-n-2,-2\}\,, \\
            -n-3 \le -4 = s-5 < -3 \le \min\{m-n-4,-2\}\,,
        \end{align*}
        so the claim follows.
    \end{proof}
\end{lemma}

\begin{remark}\label[remark]{rmk:small-m}
    A more careful analysis on small values of $m$ shows there
    are other examples of triples with similar behaviour:
    \begin{itemize}
        \item for $m=0$, consider $\{(3,0,-1), (4,0,-1), (4,0,-2), (5,0,-2)\}$;
        \item for $m=1$, consider $\{(3,1,0), (4,1,-1), (5,1,-2)\}$.
    \end{itemize}
    Then, by similar computations and comparing with \Cref{lemma:coomologico},
    one deduces that any $(n,m,s)$ listed above yields:
    \begin{itemize}
        \item $\mE^{0}$ totally acyclic, except for $(4,1,-1)$ in which satisfies $WBN_{1}$;
        \item $\mE^{-1}$ not totally acyclic, satisfying $WBN_{2}$;
        \item $\mE^{-2}$ satisfying $WBN_{3}$ and being totally acyclic for $(4,0,-1)$ and $(5,0,-2)$.
    \end{itemize}
    In particular, no pair of sheaves in their respective resolutions is simultaneously totally acyclic.
\end{remark}

\begin{lemma}\label[lemma]{lemma:coomologico.*}
    For $(n,m,s)=(1,0,-1)$ and $(n,m,s)=(2,0,-2)$,
    all the three bundles $\mE^{0}, \mE^{-1}$ or $\mE^{-2}$
    are totally acyclic.
    \begin{proof}
        In both cases it holds $s=m-1$, 
        from which one deduces $\mE^{0}$ is totally acyclic,
        and also $s-3=-n-3$, which
        implies acyclicity of both $\mE^{-1}$ and $\mE^{-2}$.
    \end{proof}
\end{lemma}

\begin{lemma}\label[lemma]{lemma:coomologico.no}
    If $(n,m,s)=(n,0,0)$, with $1 \le n \le 3$, or $(n,m,s)=(1,1,1)$, then:
    \begin{enumerate}
        \item $\mE^{0}$ satisfies $WBN_{0}$ and it is not totally acyclic;
        
        \item $\mE^{-1}$ is either totally acyclic or
        satisfies $WBN_{2}$;

        \item $\mE^{-2}$ is either totally acyclic or satisfies $WBN_{3}$.
    \end{enumerate}
    Moreover, $\mE^{-1}$ and $\mE^{-2}$ are not totally acyclic at the same time.
    \begin{proof}
        Clearly, one has $s > m-1$, hence $\mE^{0}$ satisfies
        $WBN_{0}$. On the other hand,
        \begin{equation*}
            m-n-2 \le s-3 \le -2\,,
        \end{equation*}
        hence $\mE^{-1}$ can have, at most, cohomology
        in degree $2$. If $m=0$, then $-n-3 \le  s-3 < -n$,
        hence $\mE^{-2}$ has $WBN_{3}$ but when $n=3$,
        in which case $\mE^{-2}$ is totally acyclic;
        in this case, $m-n-2=-5<s-3$, hence $\mE^{-1}$ has
        non-zero $H^{2}$.

        Finally, for $(n,m,s)=(1,1,1)$,
        then $\mE^{0}$ has $WBN_{0}$,
        $\mE^{-1}$ is totally acyclic
        and $\mE^{-2} = \mE_{1,3,-2} \oplus \mE_{1,1,-3}$,
        thus by \Cref{lemma:coomologico} one deduces
        \begin{equation*}
            H^{*}(\mE^{-2}) \simeq H^{3}(\mE_{1,1,-3}) \ne 0\,,
        \end{equation*}
        and the statement is proven.
    \end{proof}
\end{lemma}

\begin{lemma}\label[lemma]{lemma:coomologico.2}
    If $(n,m,s) \ne (0,0,0),(1,1,-1)$ 
    satisfies either one of the following two conditions:
    \begin{itemize}
        \item[a)] $s \le -n$ , or
        \item[b)] $(n,m,s)=(0,m,1)$ with $m \ge 2$,
    \end{itemize} 
    then
    \begin{enumerate}
        \item $\mE^{0}$ 
        satisfies $WBN_{p}$, with $2 \le p \le 4$.

        \item $\mE^{-1}$ 
        satisfies $WBN_{4}$;

        \item $\mE^{-2}$ 
        satisfies $WBN_{4}$.
    \end{enumerate}
    \begin{proof}
        Assume first $s < -n$. Then $s-3 < -n-3$, so by \Cref{lemma:coomologico} this shows both 
        $\mE^{-1}$ and $\mE^{-2}$ satisfy $WBN_{4}$;
        on the other hand, one notices
        $s<-n \le \max\{m-n-2, -2\}$ holds, 
        with the only exception when
        \begin{equation*}
        	s=-n-1 > -2 = \max\{m-n-2,-2\}\,,
        \end{equation*}
        which forces $n=m=0, s=-1=m-1$, hence $\mE^{0}$ is totally acyclic; 
         in all other cases, the inequality implies $\mE^{0}$ is either totally acyclic
        or satisfies $WBN_{p}$ for $2 \le p \le 4$.

        Let us now analyze the case $s=-n$.
        Then $s-3=-n-3$, so $\mE^{-1}$ is totally acyclic; 
        notice the condition
        \begin{align*}
        	s-5 < s-4 < s-3=-n-3
        \end{align*}
        holds for all $m$, thus it ensures $\mE^{-2}$ has only $H^{4}$.
        For $m \ge 1$, the inequality
        \begin{align*}
        	-n-3 < s=-n \le \max\{m-n-2,-2\}
        \end{align*}
        holds, except when $(n,m,s)=(0,1,0)$ or $(1,1,-1)$;
        by \Cref{lemma:coomologico}, this is enough to infer
        $\mE^{0}$ is either totally acyclic,
        or has non-zero cohomology either in degree $2$ or $3$.
        Notice the case $(0,1,0)$ is of the form $m-1=s$, hence
        $\mE^{0}$ is totally acyclic in this case as well. The case of $(1,1,-1)$ is handled in \cref{lemma:fail}.
       
        Finally, in case $m=0$, then both $\mE^{-2}$ and $\mE^{-1}$ are irreducible and totally acyclic because $s-3=-n-3$, 
        while 
		\begin{equation*}
			-n-2 < -n \le -2 < -1 = m-1
		\end{equation*}
		holds for all $n \ge 1$, hence
        $\mE^{0}$ is either totally acyclic or 
        has cohomology in degree $2$ (recall the case $(n,m,s)=(0,0,0)$ was excluded, hence the proof is finished).

        As last, we consider condition b),
        hence fix $n=0,s=1$ and let $m \ge 2$. 
        Then $s-3=-2=\min\{m-n-2,-2\}$,
        hence one deduces $H^{*}(\mE^{-2}) = H^{4}(\mE_{n,m-2,s-5})$ 
        and that
        $\mE^{-1}$ is totally acyclic; notice one always has
        $-2 < s \le m-n-2$, but when $m=2$, in which case $s=m-1$:
        either way, by \Cref{lemma:coomologico}
        it is possible to conclude 
        $\mE^{0}$ is either totally acyclic or satisfies $WBN_{2}$.
    \end{proof}
\end{lemma}

By combining the previous results, one obtains the following list of
$(n,m,s)$ for which $E_{n,m,s}$ has the weak Brill--Noether property.

\begin{thm}\label[thm]{thm:wbn}
    Let $E_{n,m,s}$ be as in \Cref{Enms}. 
    Then there exists the following list of conditions
    which guarantee cohomology vanishings of $E_{n,m,s}$:
    \begin{enumerate}
    	\item $E_{n,m,s}$ satisfies $WBN_{0}$ if
        $$s \ge \max\{1, m-1,m+1-n\}$$ 
        and $(n,m,s) \ne (1,1,1),(2,2,1)$;
    	\item $E_{n,m,s}$ satisfies $WBN_{1}$ 
        if $(m,n,s)$ is such that:
        \begin{itemize}
            \item either $\max\{2,m-1\} \le n \le m+3$ and $s = m-n+1$,
            \item or $(n,m,s)=(n,m,1)$, with 
            $m \ge 2$ and $\max\{1,m-2\}\le n \le m-1$,
            \item or it is an example of \cref{rmk:small-m};
        \end{itemize}
        \item $E_{n,m,s}$ satisfies $WBN_{2}$ 
        if $(n,m,s) \ne (0,0,0),(1,1,-1)$ is such that:
        \begin{itemize}
            \item either $s \le -n$,
            \item or  $(n,m,s)=(0,m,1)$, with $m \ge 2$.
        \end{itemize}
    \end{enumerate}
        In addition, the triples $(n,m,s) = (1,0,-1), (2,0,-2)$
        yield $E_{n,m,s}$ totally acyclic.
\end{thm}
\begin{proof}  
    Given $E_{n,m,s}$, the Eagon--Northcott complex 
    yields the resolution  \eqref{eqn:Eagon--Northcott},
    hence one can compute $H^{*}(X,E_{n,m,s})$ explicitly in many cases
    thanks to the spectral sequence
    \begin{equation}\label{eqn:speck}
    	E_{1}^{p,q} := H^{q}(\Gr(2,4), \mE^{p})
    	\implies H^{p+q}(X,E_{n,m,s})\,.
    \end{equation}
    Indeed, we now verify the above numerical conditions
    provide the desired vanishings.

    \begin{enumerate}
        \item The numerical conditions are the ones
        in \Cref{lemma:coomologico.0}, which
        guarantees that $H^{q}(\Gr(2,4), \mE^{p}) = 0$
        whenever $p+q \ge 1$, hence
        $H^{1}(X,E_{n,m,s}) = H^{2}(X,E_{n,m,s})=0$.
        
        \item The hypothesis are the same as in 
        \Cref{lemma:coomologico.1} and in \Cref{rmk:small-m}, 
        so the first page of the sequence~\ref{eqn:speck} has at most three non-zero terms,
        namely $E_{1}^{-2,3}, E_{1}^{-1,2}$ and $E_{1}^{0,1}$.
        Hence $E_{1}^{p,q}=E_{\infty}^{p,q}=0$ for $p+q \ne 1$,
        which is enough to show $E_{n,m,s}$ has $WBN_{1}$. 
        
        \item The conditions on $(n,m,s)$ match the ones of
        \Cref{lemma:coomologico.2}, hence we deduce that
        $H^{q}(\Gr(2,4), \mE^{p}) = 0$ for $p+q \le 1$,
        hence $E_{n,m,s}$ satisfies $WBN_{2}$.
    \end{enumerate}
    
     The claim about $(1,0,-1)$ and $(2,0,-2)$ is obvious, for \cref{lemma:coomologico.*} 
     ensures every term in the spectral sequence is zero.
\end{proof}

In light of \cref{thm:wbn}, one might wonder what happens outside the highlighted region. However, the answer is rather complicated. It is not hard to show that there exist bundles of the form $E_{n,m,s}$ that fail to have the WBN property.

\begin{example}\label[example]{lemma:fail}  
        If $(n,m,s) \in \{(1,1,1), (1,1,-1), (1,0,0), (2,0,0), (3,0,0)\}$, then $E_{n,m,s}$ fails to satisfy the weak Brill--Noether property. The triples $\{(1,0,0), (2,0,0), (3,0,0), (1,1,1)\}$
        are the ones described in \Cref{lemma:coomologico.no}:
        in these cases, the first page of the spectral sequence~\eqref{eqn:speck}
        has $E_{1}^{0,0} \ne 0$ and one between $E_{1}^{-2,3}$ and $E_{1}^{-1,2}$
        which does not vanish. As $E_{1}^{p,q} = E_{\infty}^{p,q}$,
        this is enough to conclude both $h^{0}(X,E_{n,m,s}) > 0$
        and $h^{1}(X,E_{n,m,s}) > 0$.
        
        If $(n,m,s)=(1,1,-1)$, then by \cref{lemma:coomologico}
        it follows
        \begin{align*}
        	H^{*}(\Gr(2,4), \mE^{0}) &= H^{1}(\Gr(2,4), \mE_{1,1,-1}) \ne 0 \\
        	H^{*}(\Gr(2,4), \mE^{-1}) &= 0 \\
        	H^{*}(\Gr(2,4), \mE^{-2}) &= H^{4}(\Gr(2,4), \mE_{1,1,-5}) \ne 0\,,
        \end{align*}
        so the spectral sequence~\eqref{eqn:speck} degenerates and
        one deduces $h^{1}(X,E_{n,m,s}) > 0$
        and $h^{2}(X,E_{n,m,s}) > 0$, which shows
        $E_{1,1,-1}$ fails the weak Brill--Noether property.

\end{example}

Unlike the case for which the weak Brill--Noether property holds,
analyzing vanishings of the terms in the complex \eqref{eqn:Eagon--Northcott}
is not enough to provide a whole, infinite region in $\N \times \N \times \Z$
for which the property fails, hence one should compute actual dimensions
of the cohomology spaces (see \cref{subsec:outp}).

To help the comprehension of \cref{thm:wbn} and \cref{lemma:fail}, we draw a graphic visualization of the given regions in Figures  \ref{figure:m-small}-\ref{figure:m-big}: each grey dot represents
a pair $(n,s)$ of coordinates for a triple $(n,m,s) \in \N \times \N \times \Z$, where $m$ is fixed (the corresponding value of $m$ labels each Figure).
\begin{itemize}
	\item The triples falling inside a solid grey area, labelled with $WBN_{0}$, are the $(n,m,s)$ from \cref{thm:wbn} for which the bundle $E_{n,m,s}$ has the property $WBN_{0}$;
	
	\item the triples in the dotted area, labelled with $WBN_{2}$, are the $(n,m,s)$ for which the bundle $E_{n,m,s}$ has the property $WBN_{2}$;
	
	\item the special examples marked with larger black dots
	$\bullet$
	and labelled by $WBN_{1}$ are the examples of \cref{thm:wbn}
	for which $E_{n,m,s}$ satisfies $WBN_{1}$;
	
	\item the totally acyclic cases, which occur
	for $m=0$, have been marked 
	by the letter `A' in \Cref{figure:m0};
	
	\item the counterexamples to the weak Brill--Noether property
	occurred for $m=0,1$, hence they can be seen marked
	by a cross $\times$ in Figures~{\ref{figure:m0}-\ref{figure:m1}}.
\end{itemize}

\begin{corollary}\label{cor:wbn-moduli}
    For $(n,m,s)$ in the list of \Cref{thm:wbn}, then
    the general element of the irreducible component of 
    $M_{X,H}(v(E_{n,m,s}))$ containing $E_{n,m,s}$ has at most one non-vanishing cohomology.
\end{corollary}
\begin{proof}
    If $E_{n,m,s}$ is a stable sheaf satisfying $\operatorname{WBN}_{p}$, 
    then the result is
    a consequence of the upper semicontinuity
    of the cohomology functions $h^{q}(X,-)$.
    If $E_{n,m,s}$ is strictly $H$-semistable, 
    then it is a representative of a non-trivial S-equivalence
    class $[E_{n,m,s}]$ of semistable sheaves,
    so one has to argue at the level of the stack $\mathcal{M}_{X,H}(v)$ of $H$-semistable sheaves. 
    By \cite[Theorem~{4.16.iv}]{Alp13}, 
    if $E'$ is a semistable sheaves 
    S-equivalent to $E_{n,m,s}$, then their closures in $\mathcal{M}_{X,H}(v)$ have $\overline{\{E_{n,m,s}\}} \cap \overline{\{E'\}} \ne \emptyset$.
    By \Cref{cor:polystability}, the vector bundle $E_{n,m,s}$ is polystable and it is in fact a closed point, hence
    by semicontinuity of $h^{q}(X,-)$ there is an open substack
    $\mathcal{U} \subset \mathcal{M}_{X,H}(v)$ parametrizing
    semistable sheaves satisfying $\operatorname{WBN}_{p}$.
    But then its image $U \subset M_{X,H}(v)$ via the quotient 
    map $\mathcal{M}_{X,H}(v) \to M_{X,H}(v)$ is an
    open subscheme containing $[E_{n,m,s}]$ and satisfying $\operatorname{WBN}_{p}$.
\end{proof}

\subsection{Outside the prescribed regions}\label[subsection]{subsec:outp}
We have seen in \cref{thm:wbn} that the  WBN property can be sorted nicely within certain regions, described by simple inequalities; on the other side,
\cref{lemma:fail} collects only few instances for the failure of the WBN,
thus one may wonder if it is possible to find families of counterexamples
outside these regions. However, for $(n,m,s)$ different from the ones described in  \cref{thm:wbn} and \cref{lemma:fail}, the situation has to be analyzed carefully.
In fact, upon a close analysis, it is possible to find certain special values of $(n,m,s)$ for which $E_{n,m,s}$ still has the WBN property.

\begin{example}\label[example]{prop:special-ex}
The bundle 
\[E_{0,4,2} = \Sym^{4}\mU\vert_{X}(2)\] has the $WBN_{1}$ property. By \cref{thm:BWB} one has
    \begin{align*}
        H^{*}(\Gr(2,4), \mE^{0}) &= 0\,,\\
        H^{*}(\Gr(2,4), \mE^{-1}) &= H^{2}(\Gr(2,4),\mE_{0,4,-1})
        \simeq \Sym^{2} V_{4}\,,\\
        H^{*}(\Gr(2,4), \mE^{-2}) &= 
        H^{2}(\Gr(2,4), \mE_{0,6,-1}) 
        \simeq \Sym^{4} V_{4}\,.\\
    \end{align*}
    By \eqref{eqn:speck}, there is an exact sequence
    \begin{equation*}
        0 \longrightarrow H^{0}(X,E_{0,4,2}) \longrightarrow \Sym^{4}V_{4} \overset{\psi}{\longrightarrow}
        (\Sym^{2}V_{4})^{\oplus 4} \longrightarrow H^{1}(X,E_{0,4,2}) \longrightarrow 0\,.
    \end{equation*}
The dual of the map in the middle can be identified with the map $$\psi:(a_1, a_2, a_3, a_4) \mapsto(a_1, a_2, a_3, a_4)(f_1, f_2, f_3, f_4)^t,$$ for $a_i \in \Sym^2V_4^\vee$.
The $f_i$ must be in general the data defining $X$, which is given by a generically injective morphism $\varphi: \mO^4 \to \Sym^2 \mU^\vee$. The map $\psi$ depends linearly on the equations of $f$. Hence, one can do the same computation in family, and as in \cite[Appendix B]{KMM10}, it follows that $\psi$ must be injective. Then $H^{0}(X,E_{0,4,2})$ vanishes, which allows us to conclude $E_{0,4,2}$ has $WBN_{1}$.
\end{example} 

On the other hand, it is possible to find other examples of bundles outside of regions that do not have the WBN property. Namely, we have the following:

\begin{example}\label[example]{prop:special-ex2}
The bundle \[E_{4,0,0} = \Sym^{4} \QQ \vert_{X}\] fails the WBN property. By\cref{thm:BWB}, it yields
    \begin{align*}
        H^{*}(\Gr(2,4), \mE^{0}) &= H^{0}(\Gr(2,4), \mE_{4,0,0}) \simeq \Sym^{4} V_{4} \,,\\
        H^{*}(\Gr(2,4), \mE^{-1}) &= H^{2}(\Gr(2,4), \mE_{4,0,-3}) \simeq \Sym^{2} V_{4}\,, \\
        H^{*}(\Gr(2,4), \mE^{-2}) &= H^{2}(\Gr(2,4), \mE_{4,2,-3})\simeq \Sigma_{2,2,0,0}V_{4}\,.
    \end{align*}
    Thus, regardeless of the rank of $H^{2}(\Gr(2,4), \mE_{4,2,-3}) \to H^{2}(\Gr(2,4), \mE_{4,0,-3})^{\oplus 4}$, one may infer
    \begin{equation*}
        0 \le h^{1}(\Gr(2,4), \mK) \le 20\,, \quad 20 \le h^{2}(\Gr(2,4), \mK) \le 40\,,
    \end{equation*}
    where $\mK := \ker(\mE^0 \to E_{4,0,0})$. In particular,
    from \eqref{eqn:speck} one deduces
    \begin{equation*}
        h^0(X,E_{4,0,0}) = h^{0}(\mE^0) + h^{1}(\mK) > 0\,, \quad 
        h^{1}(X,E_{4,0,0}) = h^{2}(\mK) \ge 20\,,
    \end{equation*}
    from which $E_{4,0,0}$ fails the weak Brill--Noether property.
\end{example}
The previous examples, show that each $E_{n,m,s}$ not covered by \cref{thm:wbn} should be treated on a case-by-case analysis, and might exhibit totally different behaviors. As a consequence,  \cref{thm:wbn} is probably the best general result one can obtain with these tools.

\section{About simplicity}\label{sec:simple}

\subsection{A sufficient criterion for simplicity}

Thanks to \cref{cor:polystability}, each bundle $E_{n,m,s}$ is at least polystable, hence semistable. For any given triple $(n,m,s)$, one might try to replicate the proof of \cref{lem:stability}; however, such an argument appears to be out of reach in general.

Nevertheless, one can immediately conclude that $E_{n,m,s}$ is stable whenever it is simple, i.e.\ when $h^0(\End(E_{n,m,s})) = 1$. Note that this vector space has dimension at least one, due to the decomposition
\[
\End(E_{n,m,s}) = \mathfrak{sl}(E_{n,m,s}) \oplus \mO_X.
\]

We can therefore study the cohomology of $\End(E_{n,m,s})$ using the same techniques as in \Cref{sec:wbn}.

\begin{lemma}\label[lemma]{lemma:endom}
Consider the vector bundle $\mE_{n,m,s}=\Sym^n \mQ \otimes \Sym^m \mU(s)$ on $\Gr(2,4)$. The following decomposition holds
\[
\End(\mE_{n,m,s})
\cong
\bigoplus_{\substack{0 \le i \le n \\ 0 \le j \le m}}
\mE_{2n-2i,2m-2j,n-m+j-i}.
\]
\end{lemma}

\begin{proof}
Let $E$ be a vector bundle. Then
\[
\End(E) \cong E \otimes E^{\vee}.
\]
In particular, for $E = \Sym^n \mQ \otimes \Sym^m \mU(s)$, we obtain
\begin{align*}
\End(\Sym^n \mQ \otimes \Sym^m \mU(s))
&\cong \Sym^n \mQ \otimes (\Sym^n \mQ)^{\vee}
\otimes \Sym^m \mU(s) \otimes (\Sym^m \mU(s))^{\vee} \\
&\cong \Sym^n \mQ \otimes (\Sym^n \mQ)^{\vee}
\otimes \Sym^m \mU \otimes (\Sym^m \mU)^{\vee}.
\end{align*}
Using the standard identifications $(\Sym^n \mQ)^{\vee} \cong \Sym^n \mQ(-n)$ and $(\Sym^m \mU)^{\vee} \cong \Sym^m \mU(m)$, we deduce
\[
\End(\Sym^n \mQ \otimes \Sym^m \mU(s))
\cong
\Sym^n \mQ \otimes \Sym^n \mQ
\otimes
\Sym^m \mU \otimes \Sym^m \mU(m-n).
\]

By Pieri's rule \cite[Appendix~A.1]{FH91}, we have
\[
\Sym^n \mQ \otimes \Sym^n \mQ
\cong \bigoplus_{0 \le i \le n} \Sigma_{2n-i,i}\mQ
\cong \bigoplus_{0 \le i \le n} \Sym^{2n-2i}\mQ(i),
\]
and similarly
\[
\Sym^m \mU \otimes \Sym^m \mU
\cong \bigoplus_{0 \le j \le m} \Sigma_{2m-j,j}\mU
\cong \bigoplus_{0 \le j \le m} \Sym^{2m-2j}\mU(-j).
\]
Combining these decompositions, we obtain
\[
\End(\Sym^n \mQ \otimes \Sym^m \mU(s))
\cong
\bigoplus_{\substack{0 \le i \le n \\ 0 \le j \le m}}
\Sym^{2n-2i} \mQ \otimes \Sym^{2m-2j} \mU(m - j - n + i).   \qedhere
\]
\end{proof}

This description has an immediate consequence, given in the following corollary.

\begin{corollary}\label[corollary]{cor:end}
Consider $\mE_{n,m,s}$ on $\Gr(2,4)$. Then
\[
\End(\mE_{n,m,s})
\supset
\End(\mE_{n-1,m-1,s}).
\]
In particular, if $\mE_{n,m,s}$ is not simple, then neither is 
$\mE_{n',m',s}$, for all $n' \ge n$ and $m' \ge m$.
\end{corollary}

\begin{proof}
   By \cref{lemma:endom}, we can write $\End(\mE_{n,m,s})$ as
\begin{align*}
\End(\mE_{n,m,s}) \cong &
\bigoplus_{0\leq j\leq m}
\Sym^{2n}\mQ \otimes \Sym^{2m-2j}\mU(m-j-n) \\
&
\bigoplus_{0\leq i\leq n}
\Sym^{2n-2i}\mQ \otimes \Sym^{2m}\mU(m-n+i) \\
&
\bigoplus_{\substack{0 \le i < n -1\\ 0 \le j < m-1}}
\Sym^{2n-2-2i}\mQ \otimes \Sym^{2m-2-2j}\mU(m-j-n+i).
\end{align*}

Notice that, again by \cref{lemma:endom},
\[
\End(\mE_{n-1,m-1,s})
\cong
\bigoplus_{\substack{0 \le i < n -1\\ 0 \le j < m-1}}
\Sym^{2n-2-2i}\mQ \otimes \Sym^{2m-2-2j}\mU(m-j-n+i),
\]
which proves the first part. The second part of the corollary follows by induction.
\end{proof}

\begin{thm} \label[thm]{thm:simple}
Let $X \subset \Gr(2,4)$ be a nodal Enriques surface, and let 
$\Sym^n \mQ \otimes \Sym^m \mU(s)$ be a vector bundle on $\Gr(2,4)$. 
If
\[
0 \le n \le 3 \ \text{and} \ m = 0, \quad \text{or} \quad 0 \le n \le 2 \ \text{and} \ m = 1,
\]
then the restriction $E_{n,m,s}=\Sym^n \mQ \otimes \Sym^m \mU(s)\vert_X$, for all $s \in \Z$, is simple, hence stable.

\end{thm}
\begin{proof}
    By \cref{cor:end}, it suffices to consider the pairs $(n,m)\in\{(3,0),(2,1)\}$.  
    We first consider $E_{3,0,s}=\Sym^3 \mQ(s)\vert_X$. By \cref{lemma:endom}, we obtain
    \[
    \End(E_{3,0,s})
    \cong E_{6,0,-3}
    \oplus E_{4,0,-2}
    \oplus E_{2,0,-1}
    \oplus \mathcal{O}_X.
    \]
    Therefore, to determine whether $E_{3,0,0}$ is simple, it suffices to show that, for $0\leq i\leq 3$
    \[
    H^0\bigl(X,E_{2i,0,-i}\bigr)
    \cong
    \begin{cases}
    \mathbb{C} & \text{if } i=0,\\
    0 & \text{if } i=1,2,3.
    \end{cases}
    \]
    For $i=0$, one has $H^0(X,\mathcal{O}_X)\cong \mathbb{C}$
    while, for all $i>0$, it holds $H^0(X,E_{2i,0,-i})=0$ by \cref{thm:wbn}(2).
    This proves that $E_{3,0,s}$ is simple.

    We apply the same argument to $E_{2,1,s}=\Sym^2 \mQ \otimes \mU(s)\vert_X$. By \cref{lemma:endom}, we obtain
    \begin{align*}
    \End(E_{2,1,s})
    \cong\ E_{4,2,-1}
    \oplus E_{2,2,0}
    \oplus E_{0,2,1}\oplus E_{4,0,-2}
    \oplus E_{2,0,-1}
    \oplus \mathcal{O}_X.
    \end{align*}
    Moreover, by \cref{thm:wbn}, for $0\leq i\leq 2$ and $0\leq j\leq 1$ it holds
    \[
        H^0\bigl(X,E_{2i,2j,j-i}\bigr)
        \cong
        \begin{cases}
        \mathbb{C} & \text{if } i=j=0,\\
        0 & \text{otherwise},
        \end{cases}
    \]
    thus $E_{2,1,s}$ is simple.
\end{proof}

\begin{remark}\label[remark]{rmk:simpl}

Notice that for all pairs $(n,m)$ that do not satisfy the hypotheses of \cref{thm:simple}, if we let $K := \ker(\mE^{0} \to E_{2n,2m,m-s})$, 
then the Eagon--Northcott complex yields the following exact sequence:
\begin{align*}
    0 \to H^{1}(K) \to H^{2}(\mE_{2n,2m,m-n} \otimes \Sym^{2}\mU(-3))
    \to H^{2}(\mE_{2n,2m,m-n-3})^{\oplus 4} \to H^{2}(K) \to 0 \,,
\end{align*}
\[
0 \to H^0\bigl(E_{2n,2m,s}\bigr) \to H^1(K) \to 0.
\]

In this situation, $H^1(K)$ cannot be determined by a dimension count alone, as it also depends on $H^2(K)$. In particular, one has
\[
\Ext^0\bigl(E_{n,m,s}\bigr)
\supset H^0(\mO_X) \oplus H^1(K),
\]
so $E_{n,m,s}$ is not a priori simple.

\end{remark}

We illustrate the content of \cref{rmk:simpl} with the following example.

\begin{example}\label[example]{ex1ind}
Consider $E_{4,0,s} = \Sym^4 \mQ(s)\vert_X$. By \cref{lemma:endom}, $\Sym^{8}\mQ(-4)\vert_{X}$ is a direct summand of $\End(E_{4,0,s})$. Using \eqref{eqn:Eagon--Northcott} and \cref{thm:BWB}, we obtain
\begin{itemize}
    \item $h^2(\mathcal{E}^{-2}) = h^2(\Sym^8 \mQ(-4)\otimes \Sym^2 \mU(-3)) = 360$;
    \item $h^2(\mathcal{E}^{-1}) = h^2((\Sym^8 \mQ(-7))^{\oplus4}) = 1080$;
    \item $h^0(\mathcal{E}^{0}) = h^0(\Sym^8 \mQ(-4)) = 0$ and $h^1(\mathcal{E}^{0}) = 0$.
\end{itemize}
Hence, we obtain the following exact sequences:
\[
0 \to H^1(K) \to H^2(\Sym^8 \mQ(-4)\otimes \Sym^2 \mU(-3))
\to H^2((\Sym^8 \mQ(-7))^{\oplus4}) \to H^2(K) \to 0,
\]
\[
0 \to H^0(\Sym^8 \mQ(-4)\vert_X) \to H^1(K) \to 0.
\]
Therefore,
\[
\mathrm{ext}^0(\Sym^4 \mQ(s)\vert_X)
= h^0(\mathcal{O}_X) + h^0(\Sym^8 \mQ(-4)\vert_X).
\]
Since $h^0(\Sym^8 \mQ(-4)\vert_X) = h^1(K)$, it is not possible, a priori, to conclude that $\Sym^4 \mQ(s)\vert_X$ is simple.
\end{example}

As a direct consequence of \cref{thm:wbn} and \cref{thm:simple}, we can now describe the regions where sufficient conditions for both weak Brill--Noether and stability are satisfied.

\begin{corollary}\label[corollary]{cor:stabWBN}
Let $E_{n,m,s}$ be as in \cref{Enms} and assume that
\[
0 \leq n \leq 3 \ \text{and} \ m = 0, 
\quad \text{or} \quad 
0 \leq n \leq 2 \ \text{and} \ m = 1.
\]
Then the following hold:
\begin{enumerate}
    \item If
    \[
    (n,m,s) \in \{n = 0,\ s \geq m+1\}
    \cup \{n > 0,\ s \geq 1\}
    \cup \{(0,0,0)\},
    \]
    then $E_{n,m,s}$ is a $\mu_H$-stable vector bundle satisfying $WBN_0$.
    
    \item If
    \[
    (n,m,s) \in \{(2,0,-1), (2,1,0), (3,0,-1), (3,0,-2)\},
    \]
    then $E_{n,m,s}$ is a $\mu_H$-stable vector bundle satisfying $WBN_1$.
    
    \item If
    \[
    (n,m,s) \in \{n \neq 0,\ s \leq -n\},
    \]
    then $E_{n,m,s}$ is a $\mu_H$-stable vector bundle satisfying $WBN_2$.
\end{enumerate}
\end{corollary}

\begin{remark}
The vector bundles appearing in \Cref{cor:stabWBN} are the bundles
$\mO_X(s)$, $\mU\vert_X(s)$, $\mQ\vert_X(s)$, $\Sym^2\mQ\vert_X(s)$, $\Sym^3\mQ\vert_X(s)$, $\mQ \otimes \mU\vert_X(s)$, and $\Sym^2\mQ \otimes \mU\vert_X(s)$. 
By the Hirzebruch--Riemann--Roch formula,
one can compute the Euler characteristic of the endomorphisms of a bundle $E$ on $X$ as
\begin{align*}
    \chi(X,\End(E)) &= \rk(E)^{2} \chi(X,\mO_{X}) + 2\rk(E)\ch_{2}(E) - c_{1}(E)^{2} \\
    &= 2 \rk(E) \chi(X,E) - \rk(E)^{2} - c_{1}(E)^{2}\,.
\end{align*}
In particular, we collect all the relevant data for these bundles in \Cref{tabEnd}.

\begin{table}[h]
    \centering
    \begin{tabular}{c||c|c|c|c|c|c}
   $E_{n,m,s}$  & $\rk$ & $\ch_1$ & $\chi(E_{n,m,s})$ & $\chi(\End(E_{n,m,s}))$ & $ext^{1}$ & $ext^{2}$ \\
   \hline \hline
    $\mO_{X}(s)$ & $1$ & $sH$ & $1+5s^{2}$ & $1$ & $0$ & $0$ \\
    $\mU(s)\vert_{X}$  & $2$ & $(2s-1)H$ & $10s(s-1)+4$ & $2$ & $0$ & $1$ \\
    $\mQ(s)\vert_{X}$ & $2$ & $(2s+1)H$ & $10s(s+1)$ & $-14$ & $15$ & $0$ \\ 
    $\Sym^2\mQ\vert_{X}(s)$ & $3$ & $3(s+1)H$ & $15s(s+2)$ & $-99$ & $100$ & $0$ \\ 
    $\Sym^3\mQ\vert_{X}(s)$ & $4$ & $2(2s+3)H$ & $20s(s+3)+4$ & $-344$ & $345+x$ & $x$ \\ 
    $\mQ \otimes \mU\vert_{X}(s)$ & $4$ & $4sH$ & $20s^{2}-6$ & $-64$ & $66$ & $1$ \\
    $\Sym^2\mQ \otimes \mU\vert_{X}(s)$ & $6$ & $3(2s+1)H$ & $30s(s+1)-24$ & $-414$ & $416$ & $1$
\end{tabular} 
\caption{}
\label{tabEnd}
\end{table}
The parameter $x$ in \Cref{tabEnd} reflects the indeterminacy arising from the analysis of the Eagon--Northcott complex, analogous to what is described in \cref{prop:special-ex} and \cref{ex1ind}.
\end{remark}


\bibliographystyle{alpha}


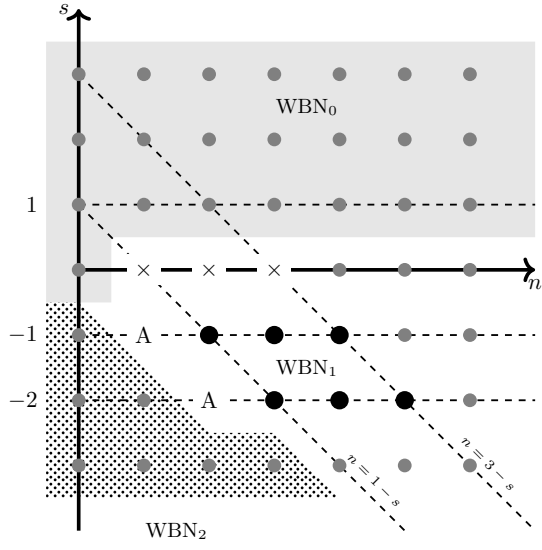
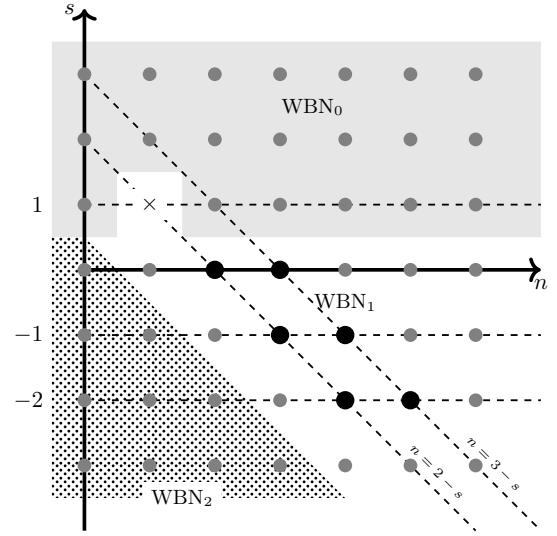
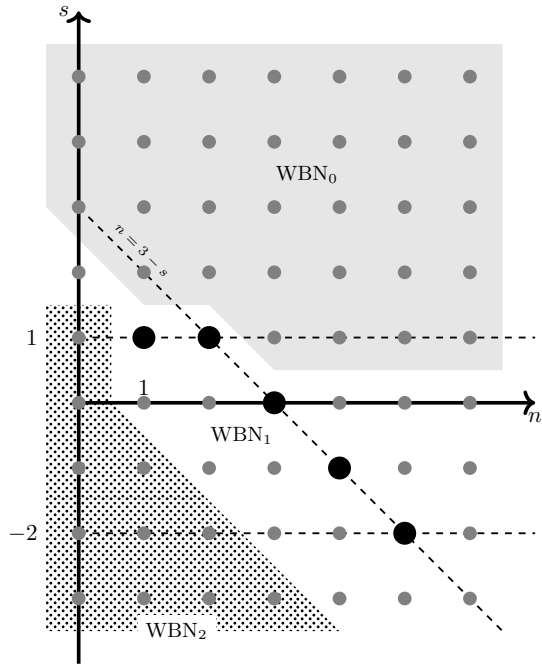
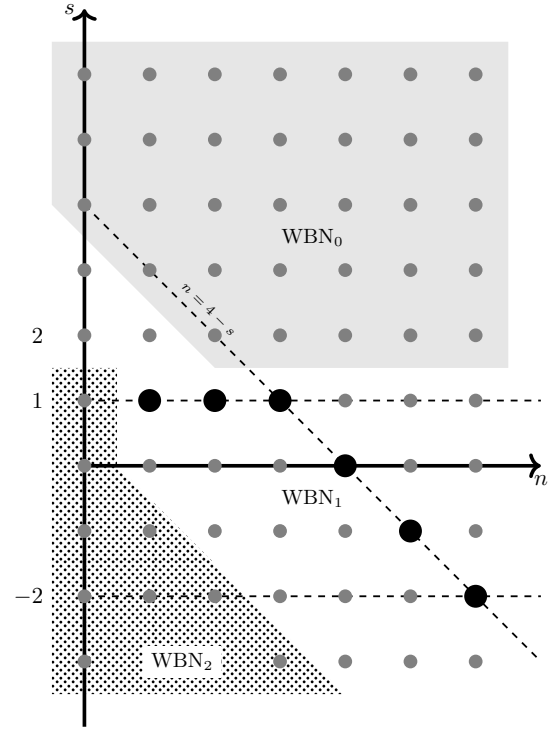
\begin{figure}[p] 
\centering

\begin{subfigure}{0.45\textwidth}
\centering
\resizebox{\linewidth}{!}{
\begin{tikzpicture}
						\fill[black!10] (-0.5,3.5) -- (-0.5,-0.5) -- (0.5,-0.5) -- (0.5,0.5) -- (7,0.5) -- (7,3.5) -- (-0.5,3.5);
						
						\fill[pattern=crosshatch dots] (-0.5,-0.5) -- (0, -0.5) -- (2,-2.5) -- (3,-2.5) -- (4,-3.5) -- (-0.5,-3.5) -- (-0.5,-0.5);
					
						\draw[ultra thick, ->] (0,-4) -- (0,4) node[left]{$s$};
						\draw[ultra thick, ->] (0,0) -- (7,0) node[below]{$n$};
						
						\draw[thick, dashed] (-0.5,1) node[left]{$1$} (0,1) -- (7,1);
						\draw[thick, dashed] (-0.5,-1) node[left, fill=white]{$-1$} (0,-1) -- (7,-1);
						\draw[thick, dashed] (-0.5,-2) node[left, fill=white]{$-2$} (0,-2) -- (7,-2);
						
					\draw[thick, dashed] (0,3) -- (7,-4) node[sloped, very near end, above]{\tiny{$n=3-s$}};
					\draw[thick, dashed] (0,1) -- (5,-4) node[sloped, above, very near end]{\tiny{$n=1-s$}};
					
						\fill[black!50] \foreach \x in {0,1,2,3,4,5,6}{
							\foreach \y in {-3,-2,-1,0,1,2,3}{
						(\x,\y) circle(3pt)
						}
					};
					
						\fill[black] (2,-1) circle(4pt);					
						\fill[black] (3,-1) circle(4pt);
						\fill[black] (4,-1) circle(4pt);
						\fill[black] (3,-2) circle(4pt);
						\fill[black] (4,-2) circle(4pt);
						\fill[black] (5,-2) circle(4pt);
						
						\draw (1,0)  node[fill=white]{$\times$};
						\draw (2,0)  node[fill=white]{$\times$};
						\draw (3,0)  node[fill=white]{$\times$};
						\draw (1,-1) node[fill=white]{A};
						\draw (2,-2) node[fill=white]{A};						
						
					\draw (3.5,2.5) node {\small{$\mathrm{WBN}_{0}$}};
					\draw (3.5,-1.5) node {\small{$\mathrm{WBN}_{1}$}};
					\draw (1.5,-4) node [fill=white]{\small{$\mathrm{WBN_{2}}$}};
\end{tikzpicture}
}
\caption{$m=0$}
\label{figure:m0}
\end{subfigure}
\hfill
\begin{subfigure}{0.45\textwidth}
\centering
\resizebox{\linewidth}{!}{
\begin{tikzpicture}
						\fill[black!10] (-0.5,3.5) -- (-0.5,0.5) -- (0.5,0.5) -- (0.5,1.5) -- (1.5,1.5)  -- (1.5,0.5) -- (7,0.5) -- (7,3.5) -- (-0.5,3.5);
						
						\fill[pattern=crosshatch dots] (-0.5,0.5) -- (0, 0.5) -- (4,-3.5) -- (-0.5,-3.5) -- (-0.5,0.5);
					
						\draw[ultra thick, ->] (0,-4) -- (0,4) node[left]{$s$};
						\draw[ultra thick, ->] (0,0) -- (7,0) node[below]{$n$};
						
						\draw[thick, dashed] (-0.5,1) node[left]{$1$} (0,1) -- (7,1);
						\draw[thick, dashed] (-0.5,-1) node[left, fill=white]{$-1$} (0,-1) -- (7,-1);
						\draw[thick, dashed] (-0.5,-2) node[left, fill=white]{$-2$} (0,-2) -- (7,-2);
						
					\draw[thick, dashed] (0,3) -- (7,-4) node[sloped, very near end, above]{\tiny{$n=3-s$}};
					\draw[thick, dashed] (0,2) -- (6,-4) node[sloped, above, very near end]{\tiny{$n=2-s$}};
					
						\fill[black!50] \foreach \x in {0,1,2,3,4,5,6}{
							\foreach \y in {-3,-2,-1,0,1,2,3}{
						(\x,\y) circle(3pt)
						}
					};
					
						\fill[black] (2,0) circle(4pt);					
						\fill[black] (3,0) circle(4pt);
						\fill[black] (3,-1) circle(4pt);
						\fill[black] (4,-1) circle(4pt);
						\fill[black] (4,-2) circle(4pt);
						\fill[black] (5,-2) circle(4pt);
						
						\draw (1,1) node[fill=white]{$\times$};					
						
					\draw (3.5,2.5) node {\small{$\mathrm{WBN}_{0}$}};
					\draw (4,-0.5) node {\small{$\mathrm{WBN}_{1}$}};
					\draw (1.5,-3.5) node [fill=white]{\small{$\mathrm{WBN_{2}}$}};
\end{tikzpicture}
}
\caption{$m=1$}
\label{figure:m1}
\end{subfigure}

\vspace{0.5cm} 

\begin{subfigure}{0.45\textwidth}
\centering
\resizebox{\linewidth}{!}{
\begin{tikzpicture}
						\fill[black!10] (-0.5,5.5) -- (-0.5,3) -- (1,1.5) -- (2,1.5) -- (3,0.5) -- (6.5,0.5) -- (6.5,5.5) -- (-0.5,5.5);
						
						\fill[pattern=crosshatch dots] (-0.5,1.5) -- (0.5,1.5) -- (0.5,0) --  (4,-3.5) -- (-0.5,-3.5) -- (-0.5,1.5);
					
						\draw[ultra thick, ->] (0,-4) -- (0,6) node[left]{$s$};
						\draw[ultra thick, ->] (0,0) -- (7,0) node[below]{$n$};
						
						\draw[thick, dashed] (-0.5,1) node[left, fill=white]{$1$} (0,1) -- (7,1);
						\draw[thick, dashed] (-0.5,-2) node[left, fill=white]{$-2$} (0,-2) -- (7,-2);
						
						\draw[thick, dashed] (0,3) -- (6.5,-3.5) node[sloped, above, very near start]{\tiny{$n=3-s$}};					
					
						\fill[black!50] \foreach \x in {0,1,2,3,4,5,6}{
							\foreach \y in {-3,-2,-1,0,1,2,3,4,5}{
						(\x,\y) circle(3pt)
						}
					};
					
						\draw (1,0) node[above]{$1$};
						\fill[black] (1,1) circle(5pt);					
						\fill[black] (2,1) circle(5pt);
						\fill[black] (3,0) circle(5pt);
						\fill[black] (4,-1) circle(5pt);
						\fill[black] (5,-2) circle(5pt);

					\draw (3.5,3.5) node {\small{$\mathrm{WBN}_{0}$}};
					\draw (2.5,-0.5) node {\small{$\mathrm{WBN}_{1}$}};
					\draw (1.5,-3.5) node [fill=white]{\small{$\mathrm{WBN_{2}}$}};
\end{tikzpicture}
}
\caption{$m=2$}
\label{figure:m2}
\end{subfigure}
\hfill
\begin{subfigure}{0.45\textwidth}
\centering
\resizebox{\linewidth}{!}{
\begin{tikzpicture}
						\fill[black!10] (-0.5,6.5) -- (-0.5,4) -- (2,1.5) -- (6.5,1.5) -- (6.5,6.5) -- (-0.5,6.5);
						
						\fill[pattern=crosshatch dots] (-0.5,1.5) -- (0.5,1.5) -- (0.5,0) --  (4,-3.5) -- (-0.5,-3.5) -- (-0.5,1.5);
					
						\draw[ultra thick, ->] (0,-4) -- (0,7) node[left]{$s$};
						\draw[ultra thick, ->] (0,0) -- (7,0) node[below]{$n$};
						
						\draw[thick, dashed] (-0.5,2) node[left]{$2$}; 
						\draw[thick, dashed] (-0.5,1) node[left, fill=white]{$1$} (0,1) -- (7,1);
						\draw[thick, dashed] (-0.5,-2) node[left, fill=white]{$-2$} (0,-2) -- (7,-2);
						
						\draw[thick, dashed] (0,4) -- (7,-3) node[sloped, above, near start]{\tiny{$n=4-s$}};
					
						\fill[black!50] \foreach \x in {0,1,2,3,4,5,6}{
							\foreach \y in {-3,-2,-1,0,1,2,3,4,5,6}{
						(\x,\y) circle(3pt)
						}
					};
					
						\fill[black] (1,1) circle(5pt);					
						\fill[black] (2,1) circle(5pt);
						\fill[black] (3,1) circle(5pt);
						\fill[black] (4,0) circle(5pt);
						\fill[black] (5,-1) circle(5pt);
						\fill[black] (6,-2) circle(5pt);

					\draw (3.5,3.5) node {\small{$\mathrm{WBN}_{0}$}};
					\draw (3.5,-0.5) node {\small{$\mathrm{WBN}_{1}$}};
					\draw (1.5,-3) node [fill=white]{\small{$\mathrm{WBN_{2}}$}};
\end{tikzpicture}
}
\caption{$m=3$}
\label{figure:m3}
\end{subfigure}

\caption{Cases $m=0,1,2,3$.}
\label{figure:m-small}
\end{figure}

\begin{figure}[htbp]
\centering

\begin{subfigure}{0.6\textwidth}
\centering
\resizebox{\linewidth}{!}{
\begin{tikzpicture}
						\fill[black!10] (-0.5,8.5) -- (-0.5,7) -- (2,4.5) -- (9.5,4.5) -- (9.5,8.5) -- (-0.5,8.5);
						
						\fill[pattern=crosshatch dots] (-0.5,1.5) -- (0.5,1.5) -- (0.5,0) --  (4,-3.5) -- (-0.5,-3.5) -- (-0.5,1.5);
					
						\draw[ultra thick, ->] (0,-4) -- (0,9) node[left]{$s$};
						\draw[ultra thick, ->] (0,0) -- (9.5,0) node[below]{$n$};
						
						\draw[thick, dashed] (-0.5,7) node[left]{$m+1$} (-0.5,5) node[left]{$m-1$} (0,5) -- (9.5,5);
						\draw[thick, dashed] (-0.5,1) node[left, fill=white]{$1$} (0,1) -- (8.5,1);
						\draw[thick, dashed] (-0.5,-2) node[left, fill=white]{$-2$} (0,-2) -- (8.5,-2);
						
					\draw[thick, dashed] (0,7) -- (9.5,-2.5) node[sloped, near start, above]{\tiny{$n=m+1-s$}};
					\draw[thick, dashed] (0,4) -- (7.5,-3.5) node[sloped, above, near start]{\tiny{$n=m-2-s$}};
					
						\fill[black!50] \foreach \x in {0,1,2,3,4,5,6,7,8,9}{
							\foreach \y in {-3,-2,-1,0,1,2,3,4,5,6,7,8}{
						(\x,\y) circle(3pt)
						}
					};
					
						\fill[black] (3,1) circle(5pt);					
						\fill[black] (4,1) circle(5pt);
						\fill[black] (5,1) circle(5pt);
						\fill[black] (6,1) circle(5pt);
						\fill[black] (7,0) circle(5pt);
						\fill[black] (8,-1) circle(5pt);
						\fill[black] (9,-2) circle(5pt);

					\draw (3.5,6.5) node {\small{$\mathrm{WBN}_{0}$}};
					\draw (6.5,-0.5) node {\small{$\mathrm{WBN}_{1}$}};
					\draw (1.5,-3) node [fill=white]{\small{$\mathrm{WBN_{2}}$}};
\end{tikzpicture}
}
\caption{$m \ge 4$}
\end{subfigure}
\caption{General case}
\label{figure:m-big}
\end{figure}
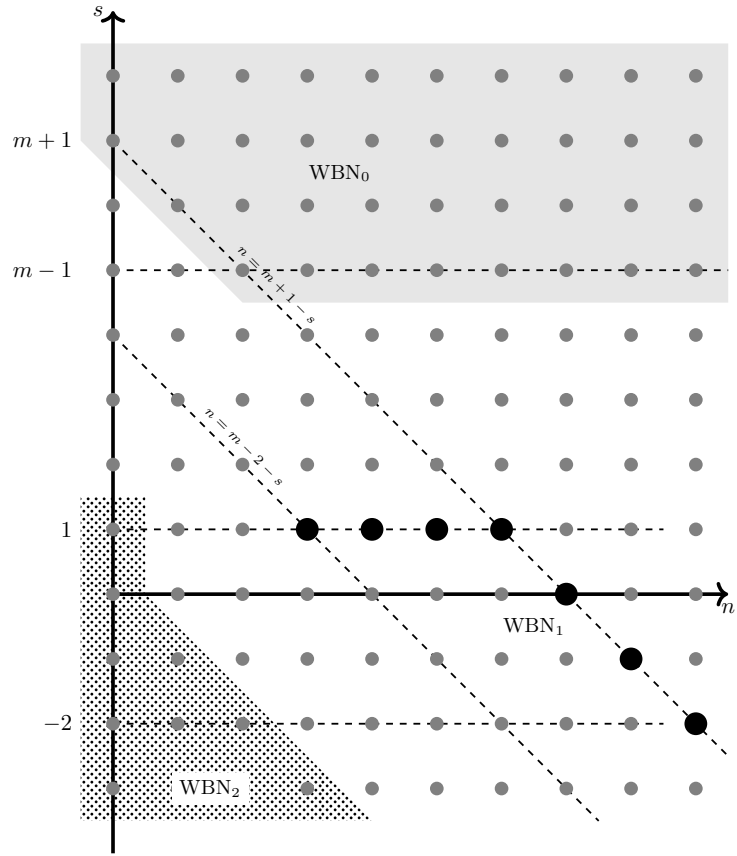

\end{document}